\documentclass[onefignum,onetabnum]{siamart190516}

\usepackage{enumitem}
\setlist[enumerate]{label=(\roman*)}
\usepackage{amsfonts}
\usepackage{amssymb}
\usepackage{mathrsfs}
\usepackage{dsfont}
\usepackage{mathtools}

\renewcommand\d{\mathrm{d}}

\newcommand\N{\mathbb{N}}

\newcommand\R{\mathbb{R}}

\DeclareMathAlphabet{\mathpzc}{OT1}{pzc}{m}{it}
\renewcommand\AA{\mathscr{A}}
\newcommand\BB{\mathscr{B}}

\newcommand\PP{\mathcal{P}}

\newcommand{\weakly}{\rightharpoonup}
\newcommand{\weaklystar}{\stackrel\star\rightharpoonup}

\DeclareMathOperator*{\essinf}{ess\,inf}

\newcommand{\sub}[1]{\underline{#1}}
\newcommand{\super}[1]{\overline{#1}}

\DeclarePairedDelimiter\paren()

\DeclarePairedDelimiter\abs{\lvert}{\rvert}
\DeclarePairedDelimiter\norm{\lVert}{\rVert}
\DeclarePairedDelimiterX\innerp[2]{(}{)}{#1,#2}
\DeclarePairedDelimiterX\dual[2]{\langle}{\rangle}{#1,#2}
\providecommand\given{\nonscript\;\delimsize|\nonscript\;}
\DeclarePairedDelimiterX\set[1]{\{}{\}}{#1}

\newsiamthm{assumption}{Assumption}
\crefname{assumption}{Assumption}{Assumptions}
\newsiamremark{remark}{Remark}

\headers{Lipschitz Stability and Directional Differentiability for QVIs}
{Constantin Christof and Gerd Wachsmuth}

\title{Lipschitz Stability and Hadamard Directional Differentiability for Elliptic and Parabolic
	Obstacle-Type Quasi-Variational Inequalities\thanks{\vspace{-\baselineskip}
\funding{This research was supported by the German Research Foundation (DFG) under grant number WA 3636/4-1
within the priority program ``Non-smooth and Complementarity-based Distributed Parameter
Systems: Simulation and Hierarchical Optimization'' (SPP 1962). The first author gratefully acknowledges 
the support by the International Research Training Group IGDK 1754 
funded by DFG and the Austrian Science Fund (FWF)
under project number 188264188/GRK1754.%
}}}

\author{Constantin Christof\thanks{%
Technische Universit\"at M\"unchen,
Faculty of Mathematics, M17,
85748 Garching bei M\"unchen, Germany,
\url{https://www-m17.ma.tum.de/Lehrstuhl/ConstantinChristof},
\email{christof@ma.tum.de}%
}%
\and
Gerd Wachsmuth\thanks{%
Brandenburgische Technische Universit\"at Cottbus-Senftenberg, 
Institute of Mathematics, 
03046 Cott\-bus, Germany, 
\url{https://www.b-tu.de/fg-optimale-steuerung},
\email{gerd.wachsmuth@b-tu.de}%
}%
}

\begin{document}

\maketitle

% REQUIRED
\begin{abstract}
This paper is concerned with the sensitivity analysis of a class of 
parameterized fixed-point problems that arise in the context of obstacle-type quasi-variational inequalities. We 
prove that, if the operators in the considered fixed-point equation satisfy a positive 
superhomogeneity condition, then the maximal and minimal element
of the solution set of the problem depend locally Lipschitz continuously on the involved parameters. 
We further show that, if certain concavity conditions hold, then the maximal solution mapping
is Hadamard directionally differentiable and its directional derivatives are precisely the minimal solutions of suitably defined 
linearized fixed-point equations. In contrast to prior results,
our analysis requires neither a Dirichlet space structure,
nor restrictive assumptions on the mapping behavior and regularity of the involved operators,
nor sign conditions on the directions that are considered in the 
directional derivatives. Our approach further covers the elliptic and parabolic 
setting simultaneously and also yields Hadamard directional differentiability results
in situations in which the solution set of the fixed-point equation is a continuum
and a characterization of 
directional derivatives via linearized auxiliary problems is provably impossible. 
To illustrate that our results can be used to study interesting problems arising in practice,
we apply them to establish the Hadamard directional differentiability of the solution 
operator of a nonlinear elliptic quasi-variational inequality, which emerges 
in impulse control and in which the obstacle mapping is obtained by taking 
essential infima over certain parts of the underlying domain,
and of the solution mapping of a parabolic quasi-variational inequality, which involves boundary controls and 
in which the state-to-obstacle relationship is described by a partial differential equation.
\end{abstract}

% REQUIRED

% REQUIRED

%%%%%%%%%%%%%%%%%%%%%%%%%%%%%%%%%%%%%%%%%%%%%%
%%%%%%%%%%%%%%%%%%%%%%%%%%%%%%%%%%%%%%%%%%%%%%

\section{Introduction and summary of results}
\label{sec:1}
The aim of this paper is to study parameterized fixed-point problems of the form
\begin{equation} 
\tag{\textup{F}}
y \in L^2(X),\quad y = S(\Phi(y), u).
\end{equation}
Here, $L^2(X)$ denotes the standard $L^2$-space on a complete measure space $(X, \Sigma, \mu)$
endowed with the partial order induced by the $\mu$-a.e.-sense;
$S\colon \bar P \times U \to L^2_+(X)$ is a function that maps elements of 
a partially ordered set $\bar P$, which possesses a largest element $\bar p$,
and a set $U$ into $L^2_+(X)$ and is nondecreasing in its first argument; 
and $\Phi\colon L^2(X) \to \bar P$ is a nondecreasing map with values in $P := \bar P \setminus \{\bar p\}$.
For the precise assumptions on the quantities in \eqref{eq:FP}, 
we refer the reader to \cref{subsec:2.1}.

Our prime interest 
is in the derivation of conditions that ensure the local Lipschitz continuity and/or directional differentiability of
certain selections from the (in general set-valued) solution mapping 
$\mathbb{S}\colon U \rightrightarrows L^2(X)$, $u \mapsto \{y \mid y = S(\Phi(y), u) \}$, associated 
with \eqref{eq:FP}. 
The main application that we have in mind is that the 
function $S$ is the solution operator of an 
elliptic or parabolic obstacle-type variational inequality,
i.e., the function that maps an obstacle $p \in \bar P$
and a right-hand side $u \in U$ to the 
solution of a variational inequality of the first kind that involves a unilateral constraint set
of the form $\{v \leq p\}$, cf.\ the examples in \cref{sec:6}.
In this situation, the fixed-point problem \eqref{eq:FP} is 
equivalent to a so-called obstacle-type quasi-variational inequality (QVI)
in which the bound defining the admissible set depends implicitly on the problem solution. 
Variational inequalities with such a structure arise, for instance, 
in the areas of mechanics, superconductivity, and 
thermoforming, see \cite{Alphonse2019-1,Alphonse2019-2,Aubin1979,Prigozhin1996-2,Prigozhin1996-1} and the references therein. 
As a prototypical example, we mention the following
elliptic quasi-variational inequality that emerges in impulse control and 
that was one of the first QVIs to be formulated when this 
problem class was introduced by Lions and Bensoussan in the nineteen-seventies,
cf.\ \cite[section VIII-2]{Bensoussan1982}: 
Given an open bounded nonempty set $\Omega \subset \R^d$, $d \in \mathbb{N}$,
a constant $\kappa \geq 0$, a function $c_0 \in L^0_+(\R^d)$,
a nondecreasing, globally Lipschitz continuous, convex function $f\colon \R \to \R$ with $f(0) = 0$
(acting as a superposition operator),
and a $u \in H^{-1}_+(\Omega) := 
\{z \in H^{-1}(\Omega) \mid \left \langle z, v\right \rangle \geq 0$ for all $0 \leq v \in H_0^1(\Omega)\}$, 
find a function $0 \leq y \in H_0^1(\Omega)$ satisfying 
\begin{equation}
\label{eq:ImpVIintro}
y \leq \Theta(y) \qquad \text{and}\qquad 
\left \langle -\Delta y + f(y) - u, v - y \right \rangle \geq 0~~\forall v \in H_0^1(\Omega), v \leq \Theta(y),
\end{equation}
where the obstacle $\Theta(y)$ is defined by 
\begin{equation}
\label{eq:thetaDefintro}
\Theta(y)(x) := \kappa + \essinf_{0 \leq \xi \in \R^d,~x + \xi \in \Omega} c_0(\xi) + y(x + \xi) \quad \text{ for a.a.\ } x \in \Omega. 
\end{equation}
Here, $\Delta\colon H_0^1(\Omega) \to H^{-1}(\Omega)$ denotes the distributional Laplacian, 
$H_0^1(\Omega)$ and $H^{-1}(\Omega)$ are defined as usual, and
$\left \langle \cdot, \cdot \right \rangle$ denotes the dual pairing. For more details on 
the above problem, its background, and its reformulation as a fixed-point equation of the form \eqref{eq:FP},
we refer the reader to \cite{Bensoussan1975,Bensoussan1975-2,Lions1986,Perthame1984,Perthame1985} and \cref{subsec:6.2}.

Due to the various processes in physics and economics that can be described by QVIs,
there has been an increasing interest in the optimal control of this class of variational inequalities,
cf.\ \cite{Adly2010,Alphonse2020-1,Dietrich2001,Wachsmuth2020} and the references therein. 
Studying optimization problems with QVI-constraints, however, turns out to be a challenging task.  
Because of the set-valuedness of the solution mapping of \eqref{eq:FP}, the formulation 
of reasonable optimal control problems for such a fixed-point equation  
typically requires working with certain distinguished
selections from the solution set (e.g., minimal and maximal elements, cf.\ \cref{th:solvability}), 
and because of the implicit and often highly nontrivial and nonsmooth dependence of the constraint set 
on the problem solution (cf.\ the function $\Theta$ in \eqref{eq:thetaDefintro}),
the derivation of necessary optimality conditions is far from straightforward for QVIs even 
in those situations where the solution set $\mathbb{S}(u)$ can be proved to be a singleton. 
Despite these difficulties, there have been several contributions in the recent years that have tried to establish 
stability and directional differentiability results for obstacle-type quasi-variational inequalities
and, by doing so, to lay the foundation for the study of optimal control problems 
governed by QVIs. 
We mention exemplarily \cite{Alphonse2019-1,Alphonse2020-2,Alphonse2020-4}, 
which establish the directional differentiability of the solution maps of 
elliptic and parabolic obstacle-type quasi-variational inequalities 
in signed (i.e., nonnegative or nonpositive) directions by means of an 
approximation argument and classical results on ordinary variational inequalities; 
\cite{Alphonse2020-1}, which proves the continuity of the minimal
and maximal solution mappings of elliptic obstacle-type QVIs in the $L^2$-spaces; 
and \cite{Wachsmuth2020}, which establishes the directional differentiability of the solution operators 
of elliptic QVIs in all directions under a smallness assumption on the obstacle mapping and by means of the results of \cite{ChristofWachsmuth2020}. 
Unfortunately, all of the above papers have in common that they require 
very restrictive and, at times, even unrealistic assumptions on the involved operators and quantities. 
See, e.g., the conditions on the sign of the directions appearing in the derivatives in 
\cite[Theorem~1]{Alphonse2019-1}, \cite[Assumption 28]{Alphonse2020-2}, and \cite[Theorems 3.6, 4.4]{Alphonse2020-4};
the assumptions on the image, the complete continuity, and the size of $\Phi$ and its derivatives in 
\cite[Assumptions (A2), (A3), (A5)]{Alphonse2019-1},
\cite[Assumptions 28, 32, 34]{Alphonse2020-2},
\cite[Theorems 3.6, 4.4]{Alphonse2020-4},
\cite[Assumption~1]{Alphonse2020-1},
and
\cite[Assumption 3.1]{Wachsmuth2020};
and the comments in \cite[Remark~2]{Adly2010}, which emphasize that compactness assumptions 
on the obstacle map are a main bottleneck in the study of obstacle-type QVIs. 
We remark that all of these conditions in particular prevent the 
differentiability results of \cite{Alphonse2019-1,Alphonse2020-4,Wachsmuth2020} 
from being applicable to the problem \eqref{eq:ImpVIintro}.  

The aim of this paper is to demonstrate that, 
if the operators $S$ and $\Phi$ possess certain 
pointwise curvature properties, which are present in many situations,
then it is possible to establish very strong Lipschitz stability and directional differentiability results 
for the solution map of a fixed-point equation of the type \eqref{eq:FP} that do not suffer from the above problems. 
Our main results are as follows. 
(See the references in brackets for the precise assumptions and statements.)\vspace{0.4em}

\begin{itemize}[itemsep=0.4em,leftmargin=0.635cm]
\item {\bf (Local Lipschitz continuity)} We show that, if 
$P$ is a subset of a real vector space,
$U$ is a subset of $L^\infty_+(Y)$ for some complete measure space $(Y, \Xi, \eta)$,
$S$ is nondecreasing in both of its arguments, 
and $S$ and $\Phi$ satisfy 
a superhomogeneity condition,
then the minimal and maximal solution map 
of \eqref{eq:FP} are locally Lipschitz continuous on the set 
$\{ u \in U \mid u \geq c \text{ for some constant } c > 0\}$ 
as functions from $L^\infty(Y)$ into all $L^q(X)$-spaces 
that $S$ maps into. (See \cref{thm:solutions_lipschitz} and \eqref{eq:randomeq273545}.)

\item {\bf (Concavity of the maximal solution operator)}
We show that, if $P$ and $U$ are 
convex subsets of real vector spaces 
and $S\colon P \times U \to L^2_+(X)$ and $\Phi\colon L^2_+(X) \to P$ are concave functions,
then the maximal solution map 
of \eqref{eq:FP} is concave and pointwise $\mu$-a.e.\ directionally differentiable.
(See \cref{prop:Mconcave}, \cref{th:monononeM}.)

\item {\bf (Hadamard directional differentiability of the maximal solution map)}
For problems \eqref{eq:FP} that are covered by both our Lipschitz continuity and 
our concavity result, we prove that the maximal solution map
is Hadamard directionally differentiable on the set $\{ u \in U \mid u \geq c \text{ for some constant } c > 0\}$
as a function from $L^\infty(Y)$ into all $L^q(X)$-spaces
that $S$ maps into. (See \cref{cor:Hadamard}.)

\item {\bf (Unique characterization of directional derivatives)}
For problems \eqref{eq:FP} that satisfy a strengthened concavity assumption
(which also ensures that $\mathbb{S}(u)$ is a singleton), 
we establish that the derivatives of the 
solution map $\mathbb{S}\colon U \to L^2_+(X)$
are uniquely characterized by the condition that they are the smallest elements 
of the solution sets of certain linearized fixed-point equations. 
(See \cref{th:auxQVI}.)\vspace{0.4em}
\end{itemize}

Note that the differentiability results in \cref{th:monononeM,cor:Hadamard} 
indeed do not require any conditions on the signs 
of the directions in the derivatives, on the smallness of $\Phi$ (or its differentiability), 
or on the existence of an underlying Dirichlet space structure, cf.\ \cite{Alphonse2019-1,Alphonse2020-4,Wachsmuth2020}. 
As we will see in \cref{sec:6}, because of this,
our theorems are in particular able to cover 
the elliptic and the parabolic setting simultaneously
and to even yield Hadamard directional differentiability results 
in situations in which the solution set $\mathbb{S}(u)$ of \eqref{eq:FP} is a continuum
and a characterization of derivatives via linearized auxiliary problems is provably impossible.
We remark that this is in stark contrast 
to, e.g., \cite{Alphonse2019-1,Wachsmuth2020}, in which the used smallness assumptions
imply that solutions of \eqref{eq:FP} are locally unique,
and that \cref{thm:solutions_lipschitz,cor:Hadamard,th:auxQVI}
improve, for instance, 
\cite[Theorem 4]{Alphonse2020-1} and
\cite[Theorem 1]{Alphonse2019-1} under \cite[Condition (A2b)]{Alphonse2019-1}.
Our theorems also seem to be the first results that allow to establish
the Hadamard directional differentiability of the 
solution operators of parabolic obstacle-type QVIs in all directions and to uniquely characterize
the associated directional derivatives.
Additional comments on this topic
can be found in the subsequent sections after the respective theorems.
Lastly, we would like to mention that
our analysis allows to prove the local Lipschitz continuity and
Hadamard directional differentiability 
of the maximal solution map $M$ of the nonlinear elliptic impulse control problem \eqref{eq:ImpVIintro} 
(which is equal to $\mathbb{S}$ in the case $\kappa>0$)
on the set $\{u \in L^\infty(\Omega) \mid u \geq c \text{ for a constant } c > 0 \}$ 
as a function from $L^\infty(\Omega)$ into all $L^q(\Omega)$-spaces, see \cref{th:impulsLinfty,th:impulseH-1}. 
Again, this paper seems to be the first contribution to accomplish this. 

\subsection*{Content of the remaining sections}
We conclude this introduction with a brief overview of the structure of the remainder of the paper.

\Cref{sec:2} is concerned with preliminaries. 
Here, we clarify the notation, state our standing assumptions, and discuss the solvability of \eqref{eq:FP}
as well as the existence and properties of smallest and largest elements of its solution set. 

\Cref{sec:3} addresses the Lipschitz stability of the minimal and maximal solution map
of \eqref{eq:FP}. See \cref{thm:solutions_lipschitz} for the main result of this section. 

In \cref{sec:4}, we prove that the maximal solution operator of \eqref{eq:FP}
is indeed concave when the functions $S$ and $\Phi$ possess this property, 
see \cref{prop:Mconcave}. 
In \cref{th:monononeM,cor:Hadamard}, 
we then study the consequences that this observation 
has for the directional differentiability properties of the maximal solution map. 

\Cref{sec:5} establishes the already mentioned characterization result 
for directional derivatives by means of a linearized auxiliary problem, see \cref{th:auxQVI}.

\Cref{sec:6} contains three examples that illustrate that our Lipschitz continuity and directional differentiability results
can be used to study interesting problems and are also relevant for applications. 
In \cref{subsec:6.1}, we begin the discussion with a simple one-dimensional model quasi-variational inequality that 
demonstrates that the Hadamard directional differentiability result in \cref{cor:Hadamard}
covers cases in which the considered QVI possesses a continuum of solutions. 
This example in particular shows that the strengthened concavity assumption in \cref{th:auxQVI} is necessary 
and that, without it, 
the characterization of derivatives via linearized auxiliary problems may be impossible. 
The subsequent \cref{subsec:6.2} is concerned with the application of our theory to the QVI \eqref{eq:ImpVIintro}.
Here, we show that \cref{thm:solutions_lipschitz,cor:Hadamard} indeed allow to prove the 
Hadamard directional differentiability of the maximal solution map of this problem, see \cref{th:impulseH-1,th:impulsLinfty}.
In \cref{subsec:6.3}, we finally consider an evolution QVI with boundary controls in which the obstacle mapping 
arises from a parabolic partial differential equation. 
This example shows that all of the results 
in \cref{sec:2,sec:3,sec:4,sec:5} are also applicable to time-dependent problems.

%%%%%%%%%%%%%%%%%%%%%%%%%%%%%%%%%%%%%%%%%%%%%%
%%%%%%%%%%%%%%%%%%%%%%%%%%%%%%%%%%%%%%%%%%%%%%

\section{Preliminaries}
\label{sec:2}
This section is concerned with preliminaries. \Cref{subsec:2.1} clarifies the notation, rigorously formulates the considered problem, 
and collects our standing assumptions. \Cref{subsec:2.2} contains results on the solvability of \eqref{eq:FP}
and the properties and existence of minimal and maximal solutions.

\subsection{Notation, standing assumptions, and problem setting}
\label{subsec:2.1}
Throughout this paper, 
$(X, \Sigma, \mu)$ denotes a complete measure space. 
The $L^q$-spaces on $(X,\Sigma, \mu)$ are denoted by $L^q(X)$, $1 \leq q \leq \infty$, 
and the vector space of equivalence classes of real-valued measurable functions on $X$ by $L^0(X)$
(so that $L^q(X) \subset L^0(X)$ for all $q \in \{0\} \cup [1, \infty]$). 
We equip the spaces $L^q(X)$, $1 \leq q \leq \infty$, with the usual norms $\|\cdot \|_{L^q(X)}$
and endow  $L^q(X)$ for all $q \in \{0\} \cup [1, \infty]$ with the partial order induced by the $\mu$-a.e.-sense,
i.e., $v_1  \geq v_2 :\iff v_1 \geq v_2~\mu\text{-a.e.\ in }X$. 
For later use, we 
also introduce the abbreviations 
$\smash{L^{[r,s]}(X) := L^r(X)\cap L^s(X)}$ for all $1 \leq r \leq s \leq \infty$
(equipped with the norm $\|\cdot\|_{L^{[r,s]}(X)} :=  \|\cdot \|_{L^r(X)} + \|\cdot \|_{L^s(X)}$),
 $L^q_+(X) := \{v \in L^q(X) \mid v \geq 0\}$ for all 
$q \in \{0\} \cup [1, \infty]$, 
$L^\infty_\oplus(X) := \{v \in L^\infty(X) \mid v \geq c \text{ for a constant } c>0\}$,
and
$\smash{L_+^{[r,s]}(X):= L_+^r(X)\cap L_+^s(X)}$. 
Note that we have $\smash{L^{[r,s]}(X) = \bigcap_{q \in [r, s]}L^q(X)}$ by H{\"o}lder's inequality. 
With $L^0(X,(-\infty, \infty])$, we denote the
set of equivalence classes of extended real-valued measurable functions on $X$ 
with values in $(-\infty, \infty]$. 
Sometimes, we require a second complete measure space, which we denote by $(Y, \Xi, \eta)$.
We use the same notations and conventions for the spaces $L^q(Y)$ as for the spaces $L^q(X)$.

In all what follows, the symbol $\bar P$ denotes a set that is equipped with a 
partial order $\leq$ and possesses a largest element $\bar p \in \bar P$.
We always assume that $\bar P$ contains at least two elements 
so that the set $P := \bar P \setminus \{\bar p\}$ is nonempty. 
Recall that 
a partial order is a binary relation that is 
reflexive ($p \leq p$ for all $p \in \bar P$), 
antisymmetric (if $p_1 \leq p_2$ and $p_2 \leq p_1$, then $p_1 = p_2$ for all $p_1, p_2 \in \bar P$), and 
transitive (if $p_1 \leq p_2$ and $p_2 \leq p_3$, then $p_1 \leq p_3$ for all $p_1, p_2, p_3 \in \bar P$) 
and that a largest element $\bar p$ of a partially ordered set $\bar P$ is an element satisfying $\bar p \in \bar P$ and 
$p \leq \bar p$ for all $p \in \bar P$. Due to the antisymmetry of $\leq$, such an element is always unique.
In our applications, 
the element $\bar p$ can be understood as $\infty$
and $P = \bar P \setminus \{\bar p\}$ is the nonnegative cone in a real vector space.
Finally, with $U$ we denote a nonempty set (the set of parameters). 

As already mentioned in 
\cref{sec:1}, the main objective of this paper is to study parameterized fixed-point equations of the form 
\begin{equation} 
\label{eq:FP}
\tag{\textup{F}}
y \in L^2(X),\quad y = S(\Phi(y), u).
\end{equation}
Our standing assumptions on the operators $S$ and $\Phi$ in \eqref{eq:FP} are as follows:\vspace{0.4em}

\begin{itemize}[itemsep=0.4em,leftmargin=0.635cm]
\item $S\colon \bar P \times U \to L^2_+(X)$ and it holds  
$u \in U, p_1, p_2 \in \bar P, p_1 \leq p_2 \Rightarrow S(p_1, u) \leq S(p_2, u)$. 
\item $\Phi\colon L^2(X) \to P$ and it holds $v_1,v_2 \in L^2(X), v_1 \leq v_2 \Rightarrow \Phi(v_1) \leq \Phi(v_2)$.\vspace{0.4em}
\end{itemize}
Tangible examples of operators $S$ and $\Phi$ that arise in the context of obstacle-type QVIs and 
satisfy the above conditions can be found in \cref{sec:6}. 
For the sake of brevity, we will sometimes also work with the shorthand notation $T_u(v) := S(\Phi(v), u)$ in this paper. 
Note that, using the map $T_u\colon L^2(X) \to L^2_+(X)$, the problem 
\eqref{eq:FP} can be recast as $y = T_u(y)$ for all $u \in U$.  
We remark that
the above standing assumptions on $(X, \Sigma, \mu)$, $\bar P$, $P$, $U$, $S$, and $\Phi$
will be complemented with additional conditions in the subsequent sections, see
\cref{ass:Lipschitz,ass:DirDiff,ass:AuxProb}. In  the statements of our theorems,
we will always make precise if such assumptions are needed. 

\subsection{Existence of solutions via an order approach}
\label{subsec:2.2}
To establish that the problem \eqref{eq:FP} possesses a nonempty set of solutions under the standing assumptions
of \cref{subsec:2.1}, one can use a classical order approach based on the theorem of Birkhoff-Tartar, see 
\cite[section 15.2.2]{Aubin1979}. 
As we need several results on minimal and maximal solutions that are obtained from this method of proof for our sensitivity analysis, 
we present the arguments in detail in this subsection.

\begin{lemma}
\label{lemma:Tincreasing}%
The map $T_u\colon L^2(X) \to L^2_+(X)$ is nondecreasing, i.e., 
for all $u \in U$ and all $v_1, v_2 \in L^2(X)$ with $v_1 \leq v_2$,
we have $T_u(v_1) \leq T_u(v_2)$.
\end{lemma}
\begin{proof}
Due to our assumptions on $\Phi$ and $S$, $v_1 \leq v_2$ implies \mbox{$\Phi(v_1) \leq \Phi(v_2)$} and 
$T_u(v_1) =  S(\Phi(v_1), u) \leq S(\Phi(v_2), u) = T_u(v_2)$.
This proves the claim. 
\end{proof}

Next, we recall two classical concepts.

\begin{definition}[sub- and supersolutions]%
Let $u \in U$ be fixed. 
A $v \in L^2(X)$ 
is called a subsolution of \eqref{eq:FP} with parameter $u$ if $v \leq  T_u(v)$ 
and a supersolution of \eqref{eq:FP} with parameter $u$ if $v \geq  T_u(v)$.
\end{definition}

Using the monotonicity properties of $S$ and $\Phi$,
the nonnegativity of $S$, and the largest element $\bar p$ of $\bar P$, we obtain
bounds on the solutions of \eqref{eq:FP}.

\begin{lemma}
\label{lemma:subsuper}
Let $u \in U$ be fixed and consider the problem \eqref{eq:FP} with parameter $u$. 
Then $S(\bar p, u) \in L^2_+(X)$ is a 
supersolution of  \eqref{eq:FP} and the zero function is a subsolution of \eqref{eq:FP}.
Further, for every subsolution $v$ of \eqref{eq:FP}, it holds $v \leq S(\bar p, u)$ and,
for every supersolution $v$ of \eqref{eq:FP}, it holds $v \geq 0$.
In particular, all solutions $y$ of \eqref{eq:FP} (should they exist) satisfy $0 \leq y \leq S(\bar p, u)$.
\end{lemma}

\begin{proof}
Due to the mapping properties of $S$ and $\Phi$, it holds 
$0 \leq S(\Phi(0), u) = T_u(0)$
and
$S(\bar p, u) \geq S(\Phi(S(\bar p, u) ), u) = T_u(S(\bar p, u))$. 
Thus, the zero function is indeed a subsolution of  \eqref{eq:FP} and $S(\bar p, u)$ is indeed a supersolution of  \eqref{eq:FP}. 
For all subsolutions $v$  of \eqref{eq:FP}, we further have 
$v \leq T_u(v) = S(\Phi(v), u) \leq S(\bar p, u)$,
and, for all supersolutions $v$ of \eqref{eq:FP}, it holds 
$v \geq T_u(v) = S(\Phi(v), u) \geq 0$.
This completes the proof. 
\end{proof}

Via the theorem of Birkhoff-Tartar, we can now prove the existence of minimal and  maximal solutions.

\begin{theorem}[solvability of \eqref{eq:FP}]%
\label{th:solvability}%
Suppose that a $u \in U$ is given and consider the problem \eqref{eq:FP} with parameter $u$. 
Then the set of solutions $\mathbb{S}(u) \subset L^2(X)$ of \eqref{eq:FP}
is nonempty.
Further, there exist unique solutions $m(u), M(u) \in \mathbb{S}(u)$ of \eqref{eq:FP} such that,
for every subsolution $v$ of \eqref{eq:FP}, it holds $v \leq M(u)$
and such that, for every supersolution $v$ of \eqref{eq:FP}, it holds $m(u) \leq v$.
In particular,
\begin{equation}
\label{eq:bshsg16}
m(u) \leq y \leq M(u)  \quad \forall y \in \mathbb{S}(u). 
\end{equation}
\end{theorem}

\begin{proof}
Since $T_u\colon L^2(X) \to L^2(X)$ is nondecreasing by \cref{lemma:Tincreasing} 
and since the functions $0$ and $S(\bar p, u)$ 
provide a sub- and a supersolution for \eqref{eq:FP} with $0 \leq S(\bar p, u)$ by \cref{lemma:subsuper}, 
the existence of solutions of \eqref{eq:FP}  
follows immediately from the theorem of Birkhoff-Tartar, see \cite[section 15.2.2, Proposition 2]{Aubin1979}.
This theorem also implies the existence of elements $m(u), M(u) \in \mathbb{S}(u)$
such that, for every solution $y$ of \eqref{eq:FP} with $0 \leq y \leq S(\bar p, u)$, we have $m(u) \leq y \leq M(u)$.
Since \cref{lemma:subsuper} yields that all $y \in \mathbb{S}(u)$ have to satisfy $0 \leq y \leq S(\bar p, u)$, 
\eqref{eq:bshsg16} and the uniqueness of $m(u)$ and $M(u)$ now follow immediately. 
Consider now an arbitrary but fixed subsolution $v \in L^2(X)$ of \eqref{eq:FP}.
Then \cref{lemma:subsuper} implies that $v \leq S(\bar p, u)$ holds, 
and we may again invoke the theorem of  Birkhoff-Tartar
to deduce that \eqref{eq:FP} admits at least one solution $y$ with $v \leq y \leq S(\bar p, u)$.
According to \eqref{eq:bshsg16}, this solution $y$ has to satisfy $y \leq M(u)$
which implies $v \leq M(u)$. All subsolutions are thus smaller than $M(u)$ as claimed. 
To prove that $m(u)$ is smaller than every supersolution, we can use the same arguments. 
\end{proof}

In the above situation, the solutions $m(u)$ and $M(u)$ are called the minimal and the maximal 
solution of \eqref{eq:FP}, respectively. Note that this notation makes sense since \eqref{eq:bshsg16} 
implies that $\mathbb{S}(u) \subset L^2(X)$ possesses a smallest and a largest element
and since, as a consequence, the minimal and the maximal element of $\mathbb{S}(u) $ are unique.
We would like to point out that \cref{th:solvability} shows that
the set of solutions of \eqref{eq:FP} can be studied as a whole and that 
it is not necessary to restrict the analysis to those solutions $y$ of \eqref{eq:FP} that satisfy 
$\sub y \leq y \leq \super y$ for the sub- and supersolutions $\sub y$ and $\super y$ 
appearing in the theorem of Birkhoff-Tartar, cf.\ \cite[Theorem~2]{Alphonse2020-1}. 
Having established the solvability of \eqref{eq:FP}, we now turn 
our attention to questions of Lipschitz stability. 

%%%%%%%%%%%%%%%%%%%%%%%%%%%%%%%%%%%%%%%%%%%%%%
%%%%%%%%%%%%%%%%%%%%%%%%%%%%%%%%%%%%%%%%%%%%%%

\section{Lipschitz continuity of the minimal and maximal solution map}
\label{sec:3}

To prove Lipschitz stability estimates for the minimal and the maximal solution 
mapping $m\colon U \to L^2(X)$ and $M\colon U \to L^2(X)$ associated with \eqref{eq:FP}, we require additional assumptions
on the involved sets, spaces, and functions.

\begin{assumption}[additional assumptions for proving local Lipschitz continuity]%
\label{ass:Lipschitz}%
In addition to the standing assumptions in \cref{subsec:2.1}, we require the following:
\begin{enumerate}
\item $P$ is a subset of a real vector space satisfying $\lambda p \in P$ for all $p \in P$, $\lambda \in (0, 1]$.
\item $U$ is a subset of $L^\infty_+(Y)$
for some complete measure space $(Y, \Xi, \eta)$
and it holds $\lambda u \in U$ for all $u \in U$, $\lambda \in (0, 1]$. 
\item $S$ satisfies $\lambda S(p, u) \leq S(\lambda p, \lambda u)$
for all $p \in P$, $u \in  U$, and $\lambda \in (0, 1]$,
and it holds $S(p, u_1) \leq S(p, u_2)$ for all $u_1, u_2 \in U$ with $u_1 \leq u_2$ and all $p \in P$.
\item\label{item:Lipschitz_4} $\Phi$ satisfies
$\lambda \Phi(v) \leq  \Phi(\lambda v)$ for all \mbox{$v \in L^2_+(X)$, $\lambda \in (0,1]$.}
\end{enumerate}
\end{assumption}

For examples of problems satisfying \cref{ass:Lipschitz}, 
we refer to \cref{sec:6}.
We remark that a superhomogeneity condition
analogous to that in \cref{ass:Lipschitz}\ref{item:Lipschitz_4}
has already been used in 
\cite[Theorem~4]{Alphonse2020-1}. Compare also with the earlier work  \cite{Laetsch1975}
in this context. 
In the situation of \cref{ass:Lipschitz}, we can employ \cref{th:solvability} to establish the 
following result.

\begin{theorem}[Lipschitz continuity of the minimal and maximal solution map]%
\label{thm:solutions_lipschitz}%
Suppose that \cref{ass:Lipschitz} holds 
and let $q \in [1, \infty]$ be an exponent satisfying $S(p, u) \in L^q(X)$ for all $p \in P$ and all $u \in U$. 
Then, for every $u \in U$ satisfying $u \geq c$ in $L^\infty(Y)$ for some constant $c > 0$
and every $v \in U$ satisfying $\|u - v\|_{L^\infty(Y)} \leq c - \rho$ for some $0 < \rho < c$, 
the minimal and maximal solutions $m(u)$, $m(v)$, $M(u)$, and $M(v)$ of 
\eqref{eq:FP} associated with $u$ and $v$ satisfy the stability estimates
\begin{equation}
\label{eq:mLipschitzEstimate}
	\norm{m(u) - m(v)}_{L^q(X)}
	\le
	\frac{1}{\rho}  \, \|m(u)\|_{L^q(X)}\, \norm{u - v}_{L^\infty(Y)}\phantom{.}
\end{equation}
and
\begin{equation}
\label{eq:MLipschitzEstimate}
	\norm{M(u) - M(v)}_{L^q(X)}
	\le
	\frac{1}{\rho} \, \|M(u)\|_{L^q(X)} \, \norm{u - v}_{L^\infty(Y)}.
\end{equation}
\end{theorem}
\begin{proof}
	Suppose that $u, v \in U$ with constants $0 < \rho < c$ as in the theorem are given.
	Set $\varepsilon := \norm{u - v}_{L^\infty(Y)} < c$
	and
	$\lambda := 1 - \varepsilon/c \in (0,1]$.
	Then it holds 
	\[
		(\lambda - 1) \, u
		\le
		(\lambda - 1) \, c
		=
		-\varepsilon
		\le
		v - u
		\le
		\varepsilon
		=
		(1 - \lambda) \, c
		\le
		\frac{1 - \lambda}{\lambda} \, c
		\le
		\paren[\Big]{ \frac1\lambda - 1 } \, u
	\]
	and, as a consequence, $\lambda \, u \le v \le  \lambda^{-1} \, u$ in $L^\infty(Y)$. 
	Define $y := M(u) \in L^2_+(X)$ and $z := M(v) \in L^2_+(X)$.
	From
	\[
		\lambda \, y
		=
		\lambda \, S(\Phi(y), u)
		\leq
		S(\lambda \, \Phi(y), \lambda \, u)
		\le
		S(\lambda \, \Phi(y), v)
		\le
		S(\Phi(\lambda \, y), v)
		=
		T_v(\lambda \, y),
	\]
	we get that
	$\lambda \, y$ is a subsolution of \eqref{eq:FP} with parameter $v$. 
	Since $z = M(v)$, this implies $\lambda \, y \le z$, see \cref{th:solvability}.
	Analogously, we obtain
	\[
		\lambda \, z
		=
		\lambda \, S(\Phi(z), v)
		\leq
		S(\lambda \, \Phi(z), \lambda \, v)
		\le
		S(\lambda \, \Phi(z), u)
		\le
		S(\Phi(\lambda \, z), u)
		=
		T_u(\lambda \, z),
	\]	
	so that 
	$\lambda \, z$ is a subsolution of \eqref{eq:FP} with parameter $u$. 
	Again by \cref{th:solvability}, this yields $\lambda z \leq M(u) = y$. In summary, we have now proved that 
$(\lambda - 1) \, y \le z - y \le ( \lambda^{-1} - 1 ) \, y$.
	Since $\abs{\lambda - 1} = 1-\lambda \le \lambda^{-1} - 1$,
	it follows
	\begin{equation}
	\label{eq:randomeq2635}
		|M(v) - M(u)|
		\le
		(\lambda^{-1} - 1) \, |M(u)| 
		=
		\frac{\varepsilon}{c - \varepsilon} \, |M(u)| 
		\le
		\frac{1}{\rho}  \, \norm{u - v}_{L^\infty(Y)}\, |M(u)|
	\end{equation}
	$\mu$-a.e.\ in $X$.
	For the minimal solutions,
	we can argue along the same lines:
	Set  $y := m(u) \in L^2_+(X)$ and $z := m(v) \in L^2_+(X)$. 
	Then the properties of $S$ and $\Phi$ yield
	\[
		\lambda S(\Phi(\lambda^{-1} \, y), v)
		\leq
		 S(\lambda\Phi(\lambda^{-1} \, y), \lambda v)
		 \le
		 S(\lambda\Phi(\lambda^{-1} \, y), u)
		 \le
		  S(\Phi(y), u)
		=
		y
	\]
	and 
	\[
		\lambda S(\Phi(\lambda^{-1} \, z), u)
		\leq
		 S(\lambda\Phi(\lambda^{-1} \, z), \lambda u)
		 \le
		 S(\lambda\Phi(\lambda^{-1} \, z), v)
		 \le
		  S(\Phi(z), v)
		=
		z.
	\]
	The functions $\lambda^{-1} y$ and $\lambda^{-1}z$ are thus supersolutions of \eqref{eq:FP} with parameters 
	$v$ and $u$, respectively, 
	and we again obtain from \cref{th:solvability} that $\lambda^{-1} y \geq m(v) = z$ and $\lambda^{-1} z \geq m(u) = y$
	holds and, as a consequence, that $(\lambda - 1) \, y \le z - y \le ( \lambda^{-1} - 1 ) \, y$. 
	This estimate and the same calculation as in \eqref{eq:randomeq2635} yield
	\begin{equation}
	\label{eq:randomeq2635-2}
		|m(v) - m(u)|
		\le
		\frac{1}{\rho}  \, \norm{u - v}_{L^\infty(Y)}\, |m(u)|
	\end{equation}
	$\mu$-a.e.\ in $X$.
	To finish the proof, it now suffices to integrate (or take the essential supremum)
	in \eqref{eq:randomeq2635} and \eqref{eq:randomeq2635-2} 
	and to use that the assumptions on $S$ and the equation \eqref{eq:FP} imply that 
	$m(u)$, $m(v)$, $M(u)$, and $M(v)$ are elements of $L^q(X)$. 
\end{proof} 

\begin{remark}~\label{rem:LipschitzComments}
\begin{enumerate}
\item The choice $q=2$ is always possible in \cref{thm:solutions_lipschitz}
by our standing assumptions on the mapping behavior of $S$, see \cref{subsec:2.1}.

\item\label{rem:LipschitzComments:ii}
 It is easy to check that the estimates \eqref{eq:mLipschitzEstimate} and \eqref{eq:MLipschitzEstimate}
imply that, for every $u \in U \cap L^\infty_{\oplus}(Y)$, there exist constants $C, r>0$ satisfying
\begin{equation}
\label{eq:randomeq273545}
	\norm{m(v_1) - m(v_2)}_{L^q(X)} + \norm{M(v_1) - M(v_2)}_{L^q(X)}
	\le
	 C\norm{v_1 - v_2}_{L^\infty(Y)}\phantom{.}
\end{equation}
for all $v_1, v_2 \in U$ with $\norm{v_i - u}_{L^\infty(Y)} \leq r$, $i=1,2$. 
The maps $m$ and $M$ are thus indeed locally Lipschitz continuous
on the set $U \cap L^\infty_{\oplus}(Y)$.

\item \Cref{thm:solutions_lipschitz} generalizes \cite[Theorem~4]{Alphonse2020-1} 
in the sense that it shows that the minimal and maximal solution map $m$ and $M$
of \eqref{eq:FP} are not only continuous as functions from 
$U \cap L^\infty_{\oplus}(Y)$ into the space $L^2(X)$ but even locally Lipschitz continuous 
as functions from $U \cap L^\infty_{\oplus}(Y)$ into every space $L^q(X)$, $1 \leq q \leq \infty$, that the operator $S$ maps into. 
\Cref{thm:solutions_lipschitz} further illustrates that this local Lipschitz stability relies solely 
on the order properties in \cref{ass:Lipschitz} and does 
not require, e.g., the assumption that there is an underlying Gelfand triple structure,
that the map $\Phi$ possesses complete continuity properties, 
or that the map $S$ is positively homogeneous,
cf.\ \cite[section 2.1, Assumption 1]{Alphonse2020-1}. 
As we will see in \cref{sec:6}, the lack of these restrictions in particular allows us to apply 
\cref{thm:solutions_lipschitz} to parabolic quasi-variational inequalities. 
Note that our proof of \cref{thm:solutions_lipschitz} is also much simpler than the one in \cite{Alphonse2020-1}.

\item\label{rem:LipschitzComments:iv}
If, in the situation of \cref{thm:solutions_lipschitz},
it is known that there exists a reflexive Banach space $V \subset L^2(X)$
such that $V$ is continuously embedded into $L^2(X)$
and such that, for every $u \in U \cap L^\infty_{\oplus}(Y)$, 
there exist constants $C, r > 0$
with $\|S(p, v)\|_V \leq C$ for all $v \in U$ with $\|u - v\|_{L^\infty(Y)} \leq r$ and all $p \in P$, 
then it follows immediately that every sequence $\{u_n\} \subset  U \cap L^\infty_{\oplus}(Y)$
which converges in $L^\infty(Y)$ to a point $u \in U \cap L^\infty_{\oplus}(Y)$
not only satisfies $m(u_n) \to m(u)$ and $M(u_n) \to M(u)$ in $L^2(X)$
but also $m(u_n) \weakly m(u)$ and $M(u_n) \weakly M(u)$ in $V$. 
Indeed, in this case, the sequences $m(u_n) = S(\Phi(m(u_n)), u_n)$ and $M(u_n) =  S(\Phi(M(u_n)), u_n)$ 
are clearly bounded 
in $V$ and the convergences $m(u_n) \weakly m(u)$ and $M(u_n) \weakly M(u)$
in $V$ are straightforwardly obtained by applying the theorem of Banach-Alaoglu
and trivial contradiction arguments. We will get back to this topic in \cref{sec:6},
where we will show that such additional convergence properties are in particular available 
for elliptic and parabolic QVIs of obstacle type. 
\end{enumerate}
\end{remark}

%%%%%%%%%%%%%%%%%%%%%%%%%%%%%%%%%%%%%%%%%%%%%%
%%%%%%%%%%%%%%%%%%%%%%%%%%%%%%%%%%%%%%%%%%%%%%

\section{Directional differentiability of $\boldsymbol{M}$ via pointwise concavity}%
\label{sec:4}%
Next, we consider questions of differentiability. 
We would like to point out that we will not approach this topic
by using concepts like polyhedricity etc.\ known from the analysis of elliptic variational inequalities, 
see \cite{Haraux1977,Mignot1976,Wachsmuth2019}, but rather by 
considering pointwise curvature properties similar to those already exploited in \cref{thm:solutions_lipschitz}. 
This approach to the sensitivity analysis of 
nonsmooth systems has already been used in \cite{Brokate2015,Christof2019parob} 
to establish directional differentiability results for solution 
operators of obstacle-type evolution variational inequalities
and is very natural for problems of the form \eqref{eq:FP}
whose solvability can also be discussed with pointwise arguments, see \cref{subsec:2.2}.
As in the last section, we require some additional assumptions for our analysis.

\begin{assumption}[additional assumptions for proving pointwise differentiability]%
\label{ass:DirDiff}%
In addition to the standing assumptions in \cref{subsec:2.1}, we require the following:
\begin{enumerate}
\item $P$ and $U$ are convex subsets of real vector spaces. 
\item $S$ satisfies 
$\lambda S(p_1, u_1) + (1 - \lambda)  S(p_2, u_2) \leq S(\lambda p_1 + (1 - \lambda)p_2, \lambda u_1 + (1 - \lambda)u_2)$
for all $p_1, p_2 \in P$, $u_1, u_2 \in  U$, and $\lambda \in [0, 1]$.
\item $\Phi$ satisfies
$\lambda \Phi(v_1) + (1 - \lambda)\Phi(v_2) \leq  \Phi(\lambda v_1 + (1 - \lambda) v_2)$ 
for all $v_1, v_2 \in L^2_+(X)$ and $\lambda \in [0,1]$.
\end{enumerate}
\end{assumption}

For examples of problems that satisfy the above conditions, we again refer the reader to \cref{sec:6}. 

\begin{proposition}[pointwise concavity of the maximal solution map $M$]%
\label{prop:Mconcave}%
Suppose that  \cref{ass:DirDiff} holds.
Then, for all $u_1, u_2 \in U$ and all $\lambda \in [0, 1]$, it is true that 
\begin{equation}
\label{eq:concaveMineq}
\lambda M(u_1) + (1 - \lambda)M(u_2) \leq M(\lambda u_1 + (1-\lambda)u_2).
\end{equation}
\end{proposition}

\begin{proof}
Let  $\lambda \in [0, 1]$ and $u_1, u_2 \in U$ be given. 
From the properties of $S$ and $\Phi$, we obtain that
\begin{equation}
\label{eq:randomMeq25}
\begin{aligned}
\lambda M(u_1) + (1 - \lambda)M(u_2) 
& =
\lambda S(\Phi(M(u_1)), u_1) + (1 - \lambda)S(\Phi(M(u_2)), u_2)  
\\
&\leq 
S(\lambda \Phi(M(u_1))+ (1 - \lambda) \Phi(M(u_2)) , \lambda u_1 + (1-\lambda)u_2) 
\\
&\leq 
S( \Phi(\lambda M(u_1) + (1 - \lambda)M(u_2)) , \lambda u_1 + (1-\lambda)u_2)
\\
&=  T_{\lambda u_1 + (1-\lambda)u_2}(\lambda M(u_1) + (1 - \lambda)M(u_2)). 
\end{aligned}
\end{equation}
The function $\lambda M(u_1) + (1 - \lambda)M(u_2)$ is thus 
a subsolution of \eqref{eq:FP}
with parameter $\lambda u_1 + (1-\lambda)u_2$ and we may invoke 
\cref{th:solvability} to arrive at \eqref{eq:concaveMineq}.
\end{proof}

As a straightforward consequence of
\cref{prop:Mconcave}, we obtain the following directional differentiability result.
\begin{theorem}[pointwise directional differentiability of the map $M$]%
\label{th:monononeM}%
Suppose that 
\cref{ass:DirDiff} holds. Assume further that a $u \in U$
and an $h \in \R^+(U - u)$ satisfying $u + \tau h \in U$ for all $\tau \in (0, \tau_0)$, $\tau_0 > 0$,
are given. Define 
\[
\delta_\tau := \frac{M(u + \tau h) - M(u)}{\tau},\qquad \tau \in (0, \tau_0).
\]
Then there exists a unique $\delta \in L^0(X,(-\infty, \infty])$ such that,
for every $\{\tau_n\} \subset (0, \tau_0)$ satisfying $\tau_n  \to 0$, it holds $\delta_{\tau_n} \to \delta$
pointwise $\mu$-a.e.\ in $X $.
\end{theorem}

\begin{proof}
Suppose that $0 < \tau_2 < \tau_1  < \tau_0$ are given. 
Then it holds
$u + \tau_2 h = (1 - \tau_2/\tau_1 )u + (\tau_2/\tau_1 )(u + \tau_1  h)$
and it follows from \eqref{eq:concaveMineq} that 
\[
M (u + \tau_2 h) \geq \left (1 - \frac{\tau_2}{\tau_1 } \right )M(u) + \frac{\tau_2}{\tau_1 }M(u + \tau_1  h).
\]
The last inequality can also be written as 
\begin{equation}
\label{eq:randomeq182663}
\delta_{\tau_2} = \frac{M (u + \tau_2 h) - M(u)}{\tau_2}\geq \frac{M(u + \tau_1  h) - M(u)}{\tau_1 } = \delta_{\tau_1}.
\end{equation}
This shows that the family $\{\delta_\tau\}$ is $\mu$-a.e.\ nonincreasing in $X$ w.r.t.\ $\tau > 0$ 
in the sense that $\delta_{\tau_2} \geq \delta_{\tau_1}$  holds $\mu$-a.e.\ in $X$ for all $0 < \tau_2 < \tau_1  < \tau_0$.
Consider now an $N \in \mathbb{N}$ with $1/N < \tau_0$ and define $\delta \in L^0(X,(-\infty, \infty])$
via
\begin{equation*}
	\delta(x) := \sup_{n \ge N} \delta_{1/n}(x)\qquad \text{for $\mu$-a.a.\ } x \in X.
\end{equation*}
Then \eqref{eq:randomeq182663} yields that $\delta_{1/n} \to \delta$ holds $\mu$-a.e.\ in $X$ for $N\leq n \to \infty$.
Now, let a sequence $\{\tau_k\} \subset (0,\tau_0)$ with $\tau_k \to 0$ be given.
For every $k \in \N$, there exists $N\leq n \in \N$ with $\tau_k \ge 1/n$.
This yields $\delta_{\tau_k} \le \delta_{1/n}$ 
and, therefore, $\delta_{\tau_k} \le \delta$ $\mu$-a.e.\ in $X$.
Moreover,
for every $N\leq n \in \N$ there exists $K \in \N$ such that $k \ge K$ implies $\tau_k \le 1/n$
and, thus, $\delta_{\tau_k} \ge \delta_{1/n}$ $\mu$-a.e.\ in $X$.
Together with the first inequality, this yields
$\delta \ge \delta_{\tau_k} \ge \delta_{1/n}$ for all $k \ge K$ $\mu$-a.e.\ in $X$
and, since $n \geq N$ was arbitrary, 
$\delta_{\tau_k} \to \delta$  $\mu$-a.e.\ in $X$ for all $\{\tau_k\} \subset (0,\tau_0)$ with $\tau_k \to 0$ as claimed.
The uniqueness of $\delta$ is trivial. 
\end{proof}

In situations in which both the Lipschitz stability result in \cref{thm:solutions_lipschitz}
and the directional differentiability result in \cref{th:monononeM} are applicable, we
immediately get Hadamard directional differentiability for the 
maximal solution map $M$ of \eqref{eq:FP}.

\begin{corollary}[Hadamard directional differentiability of the map $M$]%
\label{cor:Hadamard}%
Suppose that \cref{ass:Lipschitz,ass:DirDiff} hold and 
let $1 \leq r \leq 2 \leq s \leq \infty$ be numbers such that 
$S(p, u) \in L^{[r,s]}(X)$ holds for all $p \in P$, $u \in U$.
Then the maximal solution map $M$ of \eqref{eq:FP} is Hadamard directionally differentiable 
on the set $U \cap L^\infty_{\oplus}(Y)$ in the following sense:
For every $u \in U \cap L^\infty_{\oplus}(Y)$ and every 
$h \in \R^+(U - u)$, 
there exists a unique $M'(u; h) \in \smash{L^{[r,s]}(X)}$ such that,
for all $\{\tau_n\} \subset (0, \infty)$ 
and
 $\{h_n\} \subset \R^+(U - u)$
satisfying $\tau_n \to 0$, $\|h - h_n\|_{L^\infty(Y)} \to 0$, and $u + \tau_n h_n \in U$ for all $n$, we have 
\begin{equation}
\label{eq:Lqdirdiff}
\frac{M(u + \tau_n h_n) - M(u)}{\tau_n} \to M'(u; h) \text{ in } L^q(X) \text{ for all  } q \in [r,s]\setminus \{\infty\}
\end{equation}
and
\begin{equation}
\label{eq:weakstardirdiff}
\frac{M(u + \tau_n h_n) - M(u)}{\tau_n} \weaklystar M'(u; h) \text{ in } L^1(X)^* \text{ in the case } s = \infty.
\end{equation}
Here, \eqref{eq:weakstardirdiff} is to be understood in the sense of equation \eqref{eq:weakstardef} below.
\end{corollary}

\begin{proof}
Suppose that $u$, $h$, $\{h_n\}$, and $\{\tau_n\}$ are as in the statement of the corollary.
We assume w.l.o.g.\ that $u + \tau_n h \in U$ holds for all $n$, set $\delta_n := (M(u + \tau_n h) - M(u))/\tau_n$,
and denote with $\delta \in L^0(X,(-\infty, \infty])$ the unique limit function from \cref{th:monononeM} associated with $u$ and $h$. 
Due to \cref{th:monononeM}, we know that $\delta_n \to \delta$ holds $\mu$-a.e.\ in $X$
and, due to \cref{thm:solutions_lipschitz},
that $\norm{\delta_n}_{\smash{L^{[r,s]}}(X)} \le C$ holds for a constant $C \ge 0$.
In the case $s = \infty$, this yields $\norm{\delta}_{L^\infty(X)} \le C$, and, for all $s$, by the lemma of Fatou, 
\begin{equation*}
	\norm{\delta}_{L^q(X)}^q 
	=
	 \int_X \lim_{n \to \infty} \abs{\delta_n}^q \, \d\mu 
	\le
	\liminf_{n \to \infty}   \int_X \abs{\delta_n}^q \, \d\mu 
	\le
	C^q \quad \forall q \in [r,s]\setminus \{\infty\}. 
\end{equation*}
In summary, the above shows that, 
under the assumptions of the corollary, $\delta$ is real-valued $\mu$-a.e.\ in $X$ and can be identified 
with an element of $L^{[r,s]}(X)$.
Next, we establish \eqref{eq:Lqdirdiff} for all $q \in [r,s]\setminus \{\infty\}$.
Using that 
$ 0 \le \abs{\delta - \delta_n} = \delta - \delta_n \le \delta - \delta_1 $
holds $\mu$-a.e.\ in $X$ for all sufficiently large $n$ by \eqref{eq:randomeq182663},
the convergence $\delta_n \to \delta$ in $L^q(X)$ for all $q \in [r,s]\setminus \{\infty\}$ 
follows immediately from the dominated convergence theorem.
Using  
\cref{thm:solutions_lipschitz}
(or, more precisely, its consequence \eqref{eq:randomeq273545}),
we further get
\begin{equation}
\label{eq:randomeq263535}
\begin{aligned}
	&\norm*{ \frac{M(u + \tau_n h_n) - M(u)}{\tau_n} - \delta }_{L^q(X)}
	\\
	&\quad \leq
	\norm*{\frac{M(u + \tau_n h) - M(u)}{\tau_n} - \delta }_{L^q(X)}
	+
	\norm*{\frac{M(u + \tau_n h_n) - M(u + \tau_n h)}{\tau_n} }_{L^q(X)}
	\\
	&\quad\leq
	\norm{\delta_n - \delta}_{L^q(X)} + C \norm{h_n - h}_{L^\infty(Y)}
\end{aligned}
\end{equation}
for all large enough $n$ with some constant $C>0$.
The right-hand side of this estimate converges to zero
for $n \to \infty$.
This proves \eqref{eq:Lqdirdiff}  
with  $M'(u; h) := \delta \in \smash{L^{[r,s]}(X)}$.

It remains
to prove \eqref{eq:weakstardirdiff} in the case $s = \infty$.
For this exponent, 
we obtain from the same arguments as in \eqref{eq:randomeq263535} that the sequence
$(M(u + \tau_n h_n) - M(u))/\tau_n $ is bounded in $L^\infty(X)$ 
and pointwise $\mu$-a.e.\ convergent to $M'(u; h) := \delta \in \smash{L^{[r,\infty]}(X)}$. 
In combination with the dominated convergence theorem,
this yields 
\begin{equation}
\label{eq:weakstardef}
\lim_{n \to \infty} \int_X  z\, \left ( \frac{M(u + \tau_n h_n) - M(u)}{\tau_n} - M'(u; h) \right ) \mathrm{d}\mu 
=
0 \qquad \forall z \in L^1(X).
\end{equation}
This establishes the desired convergence
\eqref{eq:weakstardirdiff} and completes the proof. 
\end{proof}

Recall that the canonical mapping of $L^\infty(X)$ into the topological dual of $L^1(X)$
is an isometric isomorphism if (and only if) the measure $\mu$ is localizable,
see \cite[243G]{Fremlin2010}.
This is in particular the case if $\mu$ is $\sigma$-finite,
see \cite[211L]{Fremlin2010}. If one of these conditions holds, then
the space $L^1(X)^*$ in \eqref{eq:weakstardirdiff} can be replaced by $L^\infty(X)$.

Note that the choice $r=s=2$ is always allowed in \cref{cor:Hadamard} 
by our standing assumptions on the map $S$, see \cref{subsec:2.1}.
It is further easy to check that, if a Lipschitz estimate for $M$ in a reflexive Banach space $V$
is available that is continuously embedded into $L^2(X)$, then \cref{cor:Hadamard}  
implies that the difference quotients on the left-hand side of \eqref{eq:Lqdirdiff} 
also converge weakly in $V$ to $M'(u; h)$, cf.\ \cref{rem:LipschitzComments}\ref{rem:LipschitzComments:iv}. 
Lastly, we would like to point out that the arguments that we have used in this section only work 
for the maximal solution map $M$. Indeed, 
by proceeding along the lines of \eqref{eq:randomMeq25}, 
one only obtains that $\lambda m(u_1) + (1 - \lambda) m(u_2)$
is a subsolution of \eqref{eq:FP} with parameter $\lambda u_1 + (1-\lambda)u_2$
for all $u_1, u_2 \in U$ and $\lambda \in [0,1]$ and this information 
does not allow to conclude that $m$ is convex or concave, cf.\ \cref{th:solvability}. 
To obtain 
differentiability results for the minimal solution map $m$, one can use, e.g., the results of \cite{Wachsmuth2020}. 
This, however, requires far more restrictive assumptions.

%%%%%%%%%%%%%%%%%%%%%%%%%%%%%%%%%%%%%%%%%%%%%%
%%%%%%%%%%%%%%%%%%%%%%%%%%%%%%%%%%%%%%%%%%%%%%

\section{Unique characterization of directional derivatives}
\label{sec:5}

Having established the directional differentiability of the maximal solution operator $M$,
we next aim at deriving an auxiliary problem 
that uniquely characterizes the derivatives $M'(u;h)$. In view of classical results 
for differential equations, one would expect that such 
an auxiliary problem can be obtained by ``differentiating'' the left- and 
the right-hand side of \eqref{eq:FP}. Given sufficiently smooth
$S$ and $\Phi$, this means that the natural candidate for an auxiliary problem for $\delta := M'(u;h)$
is (at least formally) the fixed-point equation $\delta= \Psi'( (M(u), u); (\delta, h))$,
where $\Psi$ denotes the composition $\Psi(v, u) := S(\Phi(v), u)$. 
Unfortunately, it turns out that this linearized problem may 
fail to contain any form of useful information even in cases 
in which \cref{cor:Hadamard} is applicable and the directional 
derivatives of $M$ exist.
In fact, we have an instance of \eqref{eq:FP} for which the equation $\delta= \Psi'( (M(u), u); (\delta, h))$
reduces to $\delta = \delta$,
see \cref{subsec:6.1}. To arrive at an auxiliary problem that 
is meaningful, we have to strengthen \cref{ass:Lipschitz,ass:DirDiff}.

\begin{assumption}[additional assumptions for the characterization of derivatives]%
\label{ass:AuxProb}%
In addition to the standing assumptions in \cref{subsec:2.1}, we require the following:
\begin{enumerate}
\item\label{ass:AuxProb:i} $P$ is a convex subset of a real vector space, it holds $0 \in P$, 
and the partial order on $P$ satisfies $\lambda p_1 + (1 - \lambda) p_2 \geq \lambda p_1$ for all $p_1, p_2 \in P$ and $\lambda \in [0,1]$. 
\item\label{ass:AuxProb:ii} $U$ is a convex subset of $L^\infty_+(Y)$ for some complete measure space $(Y, \Xi, \eta)$ 
and it holds $0 \in U$.
\item\label{ass:AuxProb:iii}  $S$ satisfies 
$\lambda S(p_1, u_1) + (1 - \lambda)  S(p_2, u_2) \leq S(\lambda p_1 + (1 - \lambda)p_2, \lambda u_1 + (1 - \lambda)u_2)$
for all $p_1, p_2 \in P$, $u_1, u_2 \in  U$, and $\lambda \in [0, 1]$;
it holds $S(p, u_1) \leq S(p, u_2)$ for all $u_1, u_2 \in U$ with $u_1 \leq u_2$ and all $p \in P$;
and we have $S(p, u) \in L^\infty_+(X)$ for all $p \in P$ and all $u \in U$.
\item\label{ass:AuxProb:iv}
There exists an $\varepsilon > 0$ such that $\lambda \Phi(v_1) + (1 - \lambda)\Phi(v_2) \leq  \Phi(\lambda v_1 + (1 - \lambda) v_2)$ 
holds for all $\lambda \in [0,1]$ and all $v_1, v_2 \in L^2(X)$ with  $v_1 \geq -\varepsilon$ and $v_2 \geq -\varepsilon$.
\end{enumerate}
\end{assumption}

Note that the above conditions imply in particular
that \cref{ass:Lipschitz,ass:DirDiff} are satisfied
(as one may easily check by combining the 
properties in \cref{ass:AuxProb} with the standing assumptions from \cref{subsec:2.1}).  
\Cref{ass:AuxProb} also entails that the problem 
\eqref{eq:FP} is uniquely solvable for all $u \in U$ as the following result shows.

\begin{proposition}[unique solvability of \eqref{eq:FP}]%
\label{prop:soluniqueness}%
Suppose that \cref{ass:AuxProb} holds. Then \eqref{eq:FP} is uniquely solvable for all $u \in U$. 
In particular, it holds $m \equiv M$ on $U$.
\end{proposition}

\begin{proof}
For all $\smash{v \in L^{[2,\infty]}_+(X)}$ and all $\alpha \in [0, 1)$, 
we can find a $\beta \in (\alpha, 1)$ such that 
$- \varepsilon \leq   v\, (\alpha - \beta)/(1 - \beta) \leq 0$ holds
for the $\varepsilon > 0$ in \cref{ass:AuxProb}\ref{ass:AuxProb:iv}.
Due to the mapping properties of $\Phi$ and our assumptions on $P$ and its partial order, this implies
\[
\Phi(\alpha v )
=
\Phi \left ( 
\beta v + (1 - \beta) \left (\frac{\alpha - \beta}{1 - \beta} \right ) v
\right )
\geq 
\beta \Phi \left ( v\right )
+
 (1 - \beta)  \Phi \left ( \left (\frac{\alpha - \beta}{1 - \beta} \right ) v \right ) 
\geq 
\beta \Phi \left ( v\right )
.
\]
We may thus conclude that, for 
all $\smash{v \in L^{[2,\infty]}_+(X)}$ and all $\alpha \in [0, 1)$, there exists a $\beta \in (\alpha, 1)$
satisfying $\Phi(\alpha v ) \geq \beta \Phi( v )$. To prove the uniqueness of solutions of \eqref{eq:FP},
we can now use an argument of \cite{Laetsch1975}.  
Suppose that $u \in U$ is fixed and that \eqref{eq:FP} possesses 
two solutions, i.e., that there exist $y_1,y_2$ with $S(\Phi(y_1), u) = y_1 \neq y_2 = S(\Phi(y_2), u)$.
Then it holds $y_1, y_2 \in \smash{L^{[2,\infty]}_+(X)}$ by  the properties of $S$
and we may assume w.l.o.g.\ that $y_1 \not \leq y_2$. 
In this situation, the number $\alpha := \sup\{\gamma \in \R \mid \gamma y_1 \leq y_2\}$
has to satisfy $\alpha \in [0, 1)$ and we may choose 
a $\beta \in (\alpha, 1)$ with $\Phi(\alpha y_1 ) \geq \beta \Phi( y_1 )$ to obtain 
\begin{equation*}
\begin{aligned}
y_2= S(\Phi(y_2), u) 
\geq 
S(\Phi(\alpha y_1), u) 
\geq  
S(\beta \Phi( y_1), u)
& \geq 
\beta S( \Phi( y_1), u) + (1 - \beta) S(0, u) 
\\
&\geq
\beta S( \Phi( y_1), u) 
=
\beta y_1. 
\end{aligned}
\end{equation*}
Here, we have used the monotonicity properties of $S$ and $\Phi$ 
and the concavity and nonnegativity of $S$. 
The above implies $\beta y_1 \leq y_2$ and, by the properties of $\alpha$ and $\beta$, 
$\beta \leq \alpha < \beta$ which is impossible. 
Solutions of \eqref{eq:FP} are thus indeed unique. 
As the solvability of \eqref{eq:FP} has been established in \cref{th:solvability},
this completes the proof. 
\end{proof}

We remark that \eqref{eq:FP} can possess multiple solutions if \cref{ass:AuxProb}\ref{ass:AuxProb:iv}
only holds with $\varepsilon = 0$, see the example in \cref{subsec:6.1}.
Next, we prove an auxiliary result on the properties 
of the composition $\Psi(v, u) := S(\Phi(v), u)$.

\begin{lemma}[pointwise concavity and directional differentiability of $\Psi$]%
\label{lem:concdirdifPsi}%
Suppose that \cref{ass:AuxProb} holds.
Define 
$\Psi\colon L^2(X) \times  U \to L^2_+(X)$, $\Psi(v, u) := S(\Phi(v), u)$.
Then,
for all $\lambda \in [0,1]$, all $u_1, u_2 \in U$, 
and all $v_1, v_2 \in L^2(X)$ which satisfy  $v_1 \geq -\varepsilon$ and $v_2 \geq -\varepsilon$,
where the constant $\varepsilon > 0$ is given by \cref{ass:AuxProb}\ref{ass:AuxProb:iv}, it holds
\[
\lambda \Psi( v_1, u_1)
+
(1 - \lambda) \Psi(  v_2,  u_2)
\leq
\Psi(\lambda v_1 + (1 - \lambda) v_2, \lambda u_1 + (1 - \lambda) u_2). 
\]
Further, for all tuples $(v, u) \in L^2_+(X) \times U$ and 
$(h_1, h_2) \in \smash{L^{[2,\infty]}(X)} \times \R^+(U - u)$ satisfying 
\mbox{$u + \tau_0 h_2  \in U$}
for a $\tau_0 > 0$, there exists a unique element 
$\Psi'((v,u);(h_1, h_2))$ of the set $L^0(X,(-\infty, \infty])$
such that the difference quotients
\begin{equation}
\label{eq:Psidiffquots}
\frac{\Psi(v + \tau h_1, u + \tau h_2) - \Psi(v, u)}{\tau} \in L^2(X),\qquad \tau \in (0, \tau_0),
\end{equation}
converge $\mu$-a.e.\ to $\Psi'((v,u);(h_1, h_2))$ along every 
sequence $\{\tau_n\} \subset (0, \tau_0)$ with $\tau_n \to 0$. 
\end{lemma}

\begin{proof}
If $\lambda \in [0,1]$,  $u_1, u_2 \in U$, and $v_1, v_2 \in L^2(X)$ as in the first part of the lemma are given, then
it follows from
our assumptions on $S$ and $\Phi$ that
\begin{align*}
\Psi(\lambda v_1 + (1 - \lambda) v_2, \lambda u_1 + (1 - \lambda) u_2)
&=
S(\Phi(\lambda v_1 + (1 - \lambda) v_2), \lambda u_1 + (1 - \lambda) u_2)
\\
&\geq
S(\lambda \Phi(v_1) + (1 - \lambda) \Phi(v_2), \lambda u_1 + (1 - \lambda) u_2)
\\
&\geq
\lambda S(\Phi(v_1), u_1 )
+
(1 - \lambda) S(\Phi(v_2),  u_2)
\\
&=
\lambda \Psi(v_1, u_1 )
+
(1 - \lambda) \Psi(v_2,  u_2). 
\end{align*}
This proves the concavity of $\Psi$ on $\{v \in L^2(X) \mid v \geq -\varepsilon\} \times U$. 
To prove the pointwise $\mu$-a.e.\ directional differentiability of $\Psi$,
we can proceed  analogously to \cref{th:monononeM}.
Indeed, if  tuples $(v, u) \in L^2_+(X) \times U$ and 
$(h_1, h_2) \in \smash{L^{[2,\infty]}(X)} \times \R^+(U - u)$ as in the second part of the lemma are given,
then the nonnegativity of $v$ and the $L^\infty(X)$-regularity of $h_1$ imply 
that $v + \tau h_1 \geq - \varepsilon$ holds for all $0 < \tau < \varepsilon/ \|h_1\|_{L^\infty(X)}$.
Due to the concavity of $\Psi$ on $\{v \in L^2(X) \mid v \geq -\varepsilon\} \times U$,
this yields that 
the difference quotients in \eqref{eq:Psidiffquots}
are nonincreasing 
w.r.t.\ $\tau \in (0, \min(\tau_0, \varepsilon/ \|h_1\|_{L^\infty(X)} ))$,
cf.\ \eqref{eq:randomeq182663}.
Using exactly the same arguments as in the proof of \cref{th:monononeM},
the existence of a 
$\Psi'((v,u);(h_1, h_2))$ with the desired 
properties now follows immediately.
\end{proof}

We are now in the position to prove the main result of this section.

\begin{theorem}[unique characterization of directional derivatives]%
\label{th:auxQVI}%
Suppose that \cref{ass:AuxProb} holds and that $r \in [1,2]$ is an exponent satisfying
$S(p, u) \in L^r(X)$ for all $p \in P$, $u \in U$.
Denote
the unique solution of \eqref{eq:FP} with parameter $u \in U$ 
by $\mathbb{S}(u)$
and 
let $\Psi\colon L^2(X) \times  U \to L^2_+(X)$ be defined as in \cref{lem:concdirdifPsi}. 
Then the function $\mathbb{S}\colon U \to \smash{L_+^{[r,\infty]}(X)}$ is Hadamard directionally differentiable 
on the set $U \cap L^\infty_{\oplus}(Y)$
in the sense that the assertion of \cref{cor:Hadamard} holds for $\mathbb{S}$ with the numbers $r$ and $s = \infty$.
Further, the directional derivative $\delta := \mathbb{S}'(u; h) \in  \smash{L^{[r,\infty]}(X)}$ 
at a point $u \in U \cap L^\infty_{\oplus}(Y)$ 
in a direction $h \in \R^+(U - u)$ satisfying $u + \tau_0 h \in U$
for a $\tau_0 > 0$ is
characterized by the condition that it is the (necessarily unique) smallest element of the set
\[
			\AA
			:=
			\left \{
				\zeta \in \smash{L^{[r,\infty]}(X)}
				\left |\, 
				\begin{aligned}
					&\exists  \{\zeta_n\} \subset \smash{L^{[r,\infty]}(X)}, \{\tau_n\} \subset (0, \tau_0) \colon
					\\
					 &\zeta_n \leq \zeta_{n+1} \text{ and } \tau_{n+1} \leq \tau_n \text{ for all } n, 
					 \\
					& \tau_n \to 0\text{ for } n \to \infty,\;\zeta_n  \to \zeta \text{ in } L^q(X) 
					\text{ for all } q \in [r, \infty),
					\\
					 &\mathbb{S}(u) +\tau_n\zeta_n \text{ is supersol.\ of \eqref{eq:FP} with par.\ } u + \tau_n h
					\text{ for all }n,
					\\
					 &\Psi'((\mathbb{S}(u),u);(\zeta_n, h)) - \zeta_n \to 0~\mu\text{-a.e.\ in }X \text{ for }n \to \infty
				\end{aligned}
				\right.
			\right \}
\]
w.r.t.\ the pointwise $\mu$-a.e.\ partial order on $X$. This set $\AA$ is a superset of the set
\[
\BB :=  \big \{ \zeta \in L^{[r, \infty]}(X)~\big |~\zeta = \Psi'((\mathbb{S}(u),u);(\zeta, h))~\mu\text{-a.e.\ in }X \big \},
\]
and, if $\Psi$ is semicontinuous in the sense that,
for every $w \in U$, we have
\begin{equation}
\label{eq:pointwise_lsc}
\begin{gathered}
\{z_n\} \text{ bounded in }L^{[r, \infty]}(X), 
~~z_n \leq z_{n+1}~\forall n,~~z_n \to z \text{ in }L^q(X)~\forall q \in [r, \infty)
\\
 \Rightarrow  \limsup_{n \to \infty} \Psi(z_{n}, w) \geq \Psi(z, w)~\mu\text{-a.e.\ in }X,
\end{gathered}
\end{equation}
then it holds $\AA = \BB$ and $\delta = \mathbb{S}'(u; h)$
is the smallest element of the set $\BB$.
\end{theorem}\pagebreak

\begin{proof}
The unique solvability of \eqref{eq:FP} has been proved in \cref{prop:soluniqueness},
the $\smash{L_+^{[r,\infty]}(X)}$-regularity of $\mathbb{S}(u)$ for all $u \in U$
follows from \eqref{eq:FP} and our assumptions on $S$,
and the asserted Hadamard directional differentiability of
$\mathbb{S}$ on $U \cap L^\infty_\oplus(Y)$ 
and the $\smash{L^{[r,\infty]}(X)}$-regularity of the directional derivatives 
follow immediately from \cref{cor:Hadamard} and the identity $\mathbb{S} \equiv M$. 
It remains to prove the characterization result. 
To this end, let us assume that a $u \in U \cap L^\infty_\oplus(Y)$, 
an $h \in \R^+(U - u)$ satisfying $u + \tau_0 h \in U$ for some $\tau_0 > 0$,
and an arbitrary but fixed sequence $\{\tau_n\} \subset (0, \tau_0)$ satisfying $\tau_n \to 0$
and $\tau_{n+1} \leq \tau_n$ for all $n$
are given and define
\begin{equation}
\label{eq:delta_n_def}
\delta :=\mathbb{S}'(u; h) \in \smash{L^{[r,\infty]}(X)},\quad
\delta_n := \frac{\mathbb{S}(u + \tau_n h) - \mathbb{S}(u)}{\tau_n} \in \smash{L^{[r,\infty]}(X)}
\quad \forall n \in \mathbb{N}. 
\end{equation}
We first show that $\delta \in \AA$. 
Since $\{\tau_n\}$ is nonincreasing, since $\mathbb{S} \equiv M$ holds, and 
since \cref{cor:Hadamard} can be applied with $r$ and $s=\infty$,
we obtain from \eqref{eq:randomeq182663} that $\delta_{n} \leq \delta_{n+1}$ holds for all $n$
and from \eqref{eq:Lqdirdiff} that $\delta_n \to \delta$ holds in $L^q(X)$ for all $q \in [r,\infty)$. 
Due to the definition of 
$\delta_n$, we further have $\mathbb{S}(u) + \tau_n \delta_n = \mathbb{S}(u + \tau_n h)$
so that $\mathbb{S}(u) + \tau_n \delta_n$ is a supersolution of \eqref{eq:FP}
with parameter $u + \tau_n h$. This proves that
$\delta$, $\{\delta_n\}$, and $\{\tau_n\}$ satisfy all of the conditions in $\AA$
except for the last one. To get this last condition, we note that the concavity of
$\Psi$ on $\{v \in L^2(X) \mid v \geq -\varepsilon\} \times U$,
the nonnegativity of $\mathbb{S}(u)$, the fact that  
$\{\delta_n\}$ is bounded in  $\smash{L^{[2,\infty]}(X)}$ by \cref{thm:solutions_lipschitz},
and \eqref{eq:delta_n_def} yield that
\begin{equation}
\label{eq:randomeq263535-11}
\begin{aligned}
\delta_n 
&=
\frac{\mathbb{S}(u + \tau_n h) - \mathbb{S}(u)}{\tau_n}
=\frac{\Psi(\mathbb{S}(u + \tau_n h), u + \tau_n h) - \Psi(\mathbb{S}(u ), u ) }{\tau_n }
\\
&=
\frac{\Psi(\mathbb{S}(u) + \tau_n \delta_n, u + \tau_n h) - \Psi(\mathbb{S}(u ), u ) }{\tau_n }
\leq \Psi'((\mathbb{S}(u),u);(\delta_n, h))
\end{aligned}
\end{equation}
holds $\mu$-a.e.\ in $X$ for all large enough $n$. 
Here, we have again used that pointwise $\mu$-a.e.\ concavity implies 
that difference quotients are majorized by the directional derivatives that they
approximate, cf.\ \eqref{eq:randomeq182663}.
Since $\delta_n \leq \delta_k \leq \delta$ holds 
for all $n \leq k$ and since $\Psi$ is nondecreasing 
in its first argument by the properties of $S$ and $\Phi$, we further obtain
\begin{equation}
\label{eq:randomeq263535-22}
\begin{aligned}
\delta  \geq 
 \delta_k 
 &= 
\frac{\Psi(\mathbb{S}(u) + \tau_k \delta_k, u + \tau_k h) - \Psi(\mathbb{S}(u ), u ) }{\tau_k }
\\
&\geq 
\frac{\Psi(\mathbb{S}(u) + \tau_k \delta_n, u + \tau_k h) - \Psi(\mathbb{S}(u ), u ) }{\tau_k }
\to \Psi'((\mathbb{S}(u),u);(\delta_n, h))
\end{aligned}
\end{equation}
for all $k \geq n$, where the limit is $\mu$-a.e.\ and for $k \to \infty$, see \cref{lem:concdirdifPsi}. 
From \eqref{eq:randomeq263535-11}, \eqref{eq:randomeq263535-22}, and \cref{th:monononeM},
we obtain that
$
0 \leq \Psi'((\mathbb{S}(u),u);(\delta_n, h))  - \delta_n  
\le
\delta - \delta_n \to 0
$
holds $\mu$-a.e.\ in $X$ for $n \to \infty$.
This shows that $\delta$ is indeed an element of $\AA$.

Suppose now that $\zeta$ is an arbitrary element of $\AA$ with associated sequences $\{\zeta_n\}$
and $\{\tau_n\}$ and let $\delta_n$ and $\delta$ be defined as in \eqref{eq:delta_n_def}.
Then $\mathbb{S}(u) + \tau_n \zeta_n$ is a supersolution of \eqref{eq:FP} with parameter $u + \tau_n h$
and it holds 
$\mathbb{S}(u) + \tau_n \zeta_n \geq m(u + \tau_n h) = \mathbb{S}(u + \tau_n h) = \mathbb{S}(u) + \tau_n \delta_n$
for all $n$ by \cref{th:solvability,prop:soluniqueness}. Thus, $\delta_n \leq \zeta_n$ 
and, after passing to the limit, $\delta \leq \zeta$. This shows that $\delta$ 
is the smallest element of $\AA$. 

It remains to study the inclusions between $\AA$ and $\BB$. 
To this end, let us suppose that an element $\zeta$ of $\BB$ is given. 
Then the concavity of $\Psi$ on $\{v \in L^2(X) \mid v \geq -\varepsilon\} \times U$ 
and the nonnegativity of $\mathbb{S}(u)$
yield that, for all sufficiently small $\tau > 0$, we have 
\begin{equation}
\label{eq:randomeq27353544}
\zeta = \Psi'((\mathbb{S}(u),u);(\zeta, h)) 
\geq \frac{\Psi(\mathbb{S}(u) + \tau \zeta, u + \tau  h) - \Psi(\mathbb{S}(u), u)}{\tau }.
\end{equation}
Since 
\eqref{eq:randomeq27353544} 
can be recast as $\mathbb{S}(u) + \tau  \zeta \geq S(\Phi(\mathbb{S}(u) + \tau \zeta), u + \tau h)$
due to the identity $\mathbb{S}(u) = \Psi(\mathbb{S}(u), u) $,
it follows that $ \mathbb{S}(u) + \tau \zeta $ is a supersolution of \eqref{eq:FP}
with parameter
$u + \tau h \in U$ for all sufficiently small $\tau > 0$. 
By choosing a nonincreasing sequence $\{\tau_n\} \subset (0, \tau_0)$
that converges to zero and 
has a sufficiently small first element and by defining $\zeta_n := \zeta$,
it now follows immediately that $\zeta \in \AA$ and, thus, $\BB\subset \AA$. 

Let us now finally assume that \eqref{eq:pointwise_lsc} holds and that a $\zeta \in \AA$
with associated sequences $\{\zeta_n\}$ and $\{\tau_n\}$ is given. 
Then the monotonicity of $\{\zeta_n\}$, the convergence $\zeta_n \to \zeta$ in $L^q(X)$
for all $q \in [r, \infty)$, and the regularity $\zeta \in \smash{L^{[r, \infty]}(X)}$
yield that $\zeta_n \to \zeta$ holds $\mu$-a.e.\ in $X$, that $\zeta_1 \leq \zeta_n \leq \zeta$
holds for all $n$, and that $\{\zeta_n\}$ is bounded in $\smash{L^{[r, \infty]}(X)}$. 
In combination with \cref{lem:concdirdifPsi} and the fact that $\Psi$ is nondecreasing in its first argument  
by the properties of $S$ and $\Phi$, 
this allows us to deduce that 
\begin{equation}
\label{eq:randomeq263535-445}
\begin{aligned}
\Psi'((\mathbb{S}(u),u);(\zeta, h)) 
&=
\lim_{k \to \infty} \frac{\Psi(\mathbb{S}(u) + \tau_k \zeta, u + \tau_k h) - \Psi(\mathbb{S}(u ), u ) }{\tau_k }
\\
&\geq 
\lim_{k \to \infty} \frac{\Psi(\mathbb{S}(u) + \tau_k \zeta_n, u + \tau_k h) - \Psi(\mathbb{S}(u ), u ) }{\tau_k }
\\
&= \Psi'((\mathbb{S}(u),u);(\zeta_n, h))
\end{aligned}
\end{equation}
holds $\mu$-a.e.\ in $X$ for all $n \in \mathbb{N}$. 
Since 
the convergence $\Psi'((\mathbb{S}(u),u);(\zeta_n, h)) - \zeta_n \to 0$ \mbox{$\mu$-a.e.\ in $X$}
in the last condition of $\AA$ and the convergence 
$\zeta_n \to \zeta$ $\mu$-a.e.\ in $X$
imply that 
$\Psi'((\mathbb{S}(u),u);(\zeta_n, h)) \to \zeta$
holds $\mu$-a.e.\ in $X$, \eqref{eq:randomeq263535-445}
yields $\Psi'((\mathbb{S}(u),u);(\zeta, h))  \geq \zeta$ $\mu$-a.e.\ in $X$.
To show that we also have the reverse inequality, we note that 
the convergence $\Psi'((\mathbb{S}(u),u);(\zeta_n, h)) \to \zeta$ $\mu$-a.e.\ in $X$, 
the same majorization argument as in \eqref{eq:randomeq263535-11} and \eqref{eq:randomeq27353544},
the properties of $\{\zeta_n\}$,
and \eqref{eq:pointwise_lsc} imply that
\begin{equation}
\label{eq:randomeq263535-44}
\begin{aligned}
\zeta  
= 
\lim_{n \to \infty} \Psi'((\mathbb{S}(u),u);(\zeta_n, h)) 
&\geq
\limsup_{n \to \infty}
\frac{\Psi(\mathbb{S}(u) + \tau_l \zeta_n, u + \tau_l h) - \Psi(\mathbb{S}(u ), u )}{\tau_l}
\\
&\geq
\frac{\Psi(\mathbb{S}(u) + \tau_l \zeta, u + \tau_l h) - \Psi(\mathbb{S}(u ), u )}{\tau_l} 
\end{aligned}
\end{equation}
holds for all sufficiently large $l \in \mathbb{N}$. 
Here, all limits etc.\ have to be understood in the $\mu$-a.e.-sense. 
By letting $l$ go to infinity  in \eqref{eq:randomeq263535-44},
we obtain that
$\zeta \geq \Psi'((\mathbb{S}(u),u);(\zeta, h))$
and, thus, 
that $\zeta \in L^{[r, \infty]}(X)$, $\zeta = \Psi'((\mathbb{S}(u),u);(\zeta, h))$ $\mu$-a.e.\ in $X$,
and $\zeta \in \BB$. 
This shows that $\AA =\BB$ holds under condition
\eqref{eq:pointwise_lsc} and completes the proof.
\end{proof}

By our standing assumptions, we can always choose $r=2$ in \cref{th:auxQVI}, see \cref{subsec:2.1}.
If the measure space $(X, \Sigma, \mu)$ is finite, then the $L^\infty$-regularity of $S$ in \cref{ass:AuxProb}\ref{ass:AuxProb:iii}
and H{\"o}lder's inequality
imply that \cref{th:auxQVI} can be invoked with $r=1$. 
(In this case, we have $\smash{ L^{[r, \infty]}(X)} = L^\infty(X)$ for all $r \in [1,\infty]$.) 
Note that, due to the last condition in the definition of $\AA$
and the inclusion $\BB \subset \AA$,
the set $\AA$ can be interpreted as a set of generalized solutions of the QVI $\zeta = \Psi'((\mathbb{S}(u),u);(\zeta, h))$. 
Under the pointwise semicontinuity condition \eqref{eq:pointwise_lsc}, this set of ``limiting'' solutions coincides with 
the ordinary solution set $\BB$. Finally, we would like to point out that the monotonicity properties 
of $\Psi$ and $\{z_n\}$ yield that the right-hand side of \eqref{eq:pointwise_lsc} can be replaced by
``$\Psi(z_n,w) \to \Psi(z, w)$ $\mu$-a.e.\ in $X$'' without changing the strength of 
this requirement. The semicontinuity condition \eqref{eq:pointwise_lsc} is thus, in fact, 
an assumption on the pointwise $\mu$-a.e.\ continuity of $\Psi$ along certain sequences.

%%%%%%%%%%%%%%%%%%%%%%%%%%%%%%%%%%%%%%%%%%%%%%
%%%%%%%%%%%%%%%%%%%%%%%%%%%%%%%%%%%%%%%%%%%%%%

\section{Applications and examples}%
\label{sec:6}%

In this section, we illustrate by means of three examples that 
our Hadamard directional differentiability and Lipschitz continuity results 
can be applied to many interesting problems.
We begin with a counterexample which demonstrates that 
\cref{thm:solutions_lipschitz,prop:Mconcave,th:monononeM,cor:Hadamard}
also cover situations in which 
\eqref{eq:FP} possesses a continuum of solutions
and that the linearized fixed-point equations $\delta= \Psi'( (M(u), u); (\delta, h))$
may indeed fail to contain any form of useful information.

%%%%%%%%%%%%%%%%%%%%%%%%%%%%%%%%%%%%%%%%%%%%%%

\subsection{One-dimensional example}%
\label{subsec:6.1}%
Define $X := Y := \{0\}$, $\Sigma := \Xi :=  \PP(X)$, and $\mu := \eta := \delta_0$,
where $\PP(X)$ denotes the power set of $X$ and $\delta_0$
the Dirac measure at zero. Then the spaces $(L^q(X), \|\cdot\|_{L^q(X)})$ and $(L^q(Y), \|\cdot\|_{L^q(Y)})$, $1 \leq q \leq \infty$,
can be identified with $(\R, |\cdot|)$ and the partial order on $L^q(X)$ and $L^q(Y)$ is just the usual 
one on $\R$. We further set $\bar P := [0, \infty]$ endowed with its canonical order,
$\bar p := \infty$, $P := [0,\infty)$, and $U := [0, \infty)$. It is easy to check that 
these choices satisfy all of the assumptions on $(X, \Sigma, \mu)$, $(Y, \Xi, \eta)$, $\bar P$, $P$, and $U$
in \cref{subsec:2.1} and \cref{ass:Lipschitz,ass:DirDiff,ass:AuxProb}.
As the map $S$, we consider the solution operator $S\colon [0, \infty] \times [0, \infty) \to [0, \infty)$, $(p, u) \mapsto y$,
of the one-dimensional obstacle-type variational inequality 
\[
y \in \R, \; y \leq p, \qquad y( v - y) \geq u(v - y)\quad \forall v \in \R, \; v \leq p,
\]
i.e.,
$S(p, u) = \min(p, u)$ for all $(p, u) \in [0, \infty] \times [0, \infty)$. 
The map $\Phi\colon \R \to [0, \infty)$ is chosen as an arbitrary but fixed real-valued function satisfying 
\begin{equation}
\label{eq:concaveex1d}
\Phi(s) = 0 \text{ for } s < 0,\qquad \Phi(s)  = s  \text{ for } s \in [0, 1],\qquad \Phi(s)  < s \text{ for } s > 1,
\end{equation}
which is nondecreasing on $\R$ and concave and continuously differentiable on $(0, \infty)$. 
(Such a $\Phi$ can be constructed easily.)
It is straightforward to verify that these $S$ and $\Phi$ 
satisfy all of the conditions in \cref{subsec:2.1} and \cref{ass:Lipschitz,ass:DirDiff,ass:AuxProb}
except for the concavity of $\Phi$ on the set $[-\varepsilon, \infty)$ for some 
$\varepsilon > 0$ in \cref{ass:AuxProb}\ref{ass:AuxProb:iv}.
Note that it is not possible to 
recover this strengthened concavity condition by redefining 
$\Phi$ on the set $(-\infty,0)$ since the required monotonicity and nonnegativity of $\Phi$
and the identity $\Phi(0) = 0$ already imply $\Phi(s) = 0$ for all $s \leq 0$.
Consider now the QVI
\begin{equation}
\label{eq:1dQVI}
y \in \R, \; y \leq \Phi(y), \qquad y( v - y) \geq u(v - y)\quad \forall v \in \R, \; v \leq \Phi(y).
\end{equation} 
Then the solution set $\mathbb{S}(u)$ of \eqref{eq:1dQVI} for a given $u \in U$ is identical 
to the set of solutions of the fixed-point equation $y = S(\Phi(y), u)$, which is precisely 
of the form \eqref{eq:FP}, and it is easy to check that 
$\mathbb{S}(u) = [0, \min(u, 1)]$ holds for all $u \in U$. 
In particular, $m(u) = 0$ and $M(u) = \min(u, 1)$ for all $u \in U$. 
Note that these explicit formulas show 
that the functions $m\colon U \to [0, \infty)$ and $M\colon U \to [0, \infty)$ 
are locally Lipschitz continuous on $(0, \infty)$ and that $M$ is concave and Hadamard 
directionally differentiable on $(0, \infty)$ as predicted by \cref{thm:solutions_lipschitz,prop:Mconcave,cor:Hadamard}.
As \eqref{eq:1dQVI} satisfies all of the 
assumptions in \cref{sec:2,sec:3,sec:4}, this demonstrates that 
our analysis indeed covers situations in which the 
considered QVI (or, more generally, fixed-point problem)
possesses a continuum of solutions
and the maximal and minimal element of the solution set are not isolated.
This is in contrast to the results of
\cite{Alphonse2019-1,Wachsmuth2020} which only apply to problems with locally unique solutions,
see \cite[Theorem~4.2]{Wachsmuth2020}.
The example \eqref{eq:1dQVI} also shows that the assumption of concavity of $\Phi$
on an interval of the form $[-\varepsilon, \infty)$ for some $\varepsilon > 0$
in \cref{ass:AuxProb}\ref{ass:AuxProb:iv} cannot be dropped 
in the uniqueness result of \cref{prop:soluniqueness}.
Next, we demonstrate that the same is true for the 
characterization of the directional derivatives 
of $M$ by means of the linearized 
auxiliary problems in \cref{th:auxQVI}. 
If we naively differentiate 
the fixed-point equation $y = \min(\Phi(y), u)$ associated with \eqref{eq:1dQVI} at some $y, u>0$,
then we arrive at the problem
\begin{equation}
\label{eq:candcharprob}
	\delta = \Psi'((y,u);(\delta, h)) = 
	\begin{cases}
		\Phi'(y) \, \delta& \text{for } \Phi(y) < u \\
		\min(\Phi'(y)\delta, h) & \text{for } \Phi(y) = u \\
		h & \text{for } \Phi(y) > u.
	\end{cases}
\end{equation}
Here,
$\Psi\colon \R \times  [0, \infty) \to [0, \infty)$ again denotes the composition $\Psi(v, u) := S(\Phi(v), u)$.
For all  $u \in (1, \infty)$, the identity \eqref{eq:candcharprob} simplifies for  
$y = M(u) = \min(u, 1) = 1$ to the degenerate equation
$\delta = \delta$
which does not contain any form of useful information. This shows that it can, in general,
not be expected that the directional derivatives of the maximal solution map $M$ 
of a fixed-point problem of the form \eqref{eq:FP} can be characterized 
by means of the equations $\delta= \Psi'( (M(u), u); (\delta, h))$
and that such a failure of the linearized auxiliary problems may occur even in situations in which 
the differentiability results in \cref{th:monononeM,cor:Hadamard} are applicable and 
the directional derivatives of $M$ exist. In particular,
the assumption that $\Phi$ is concave on a set slightly larger than $L^2_+(X)$
in \cref{ass:AuxProb}\ref{ass:AuxProb:iv} cannot be dropped for \cref{th:auxQVI} to be true. 
Note that the function $\Phi$ in \eqref{eq:concaveex1d} trivially satisfies \eqref{eq:pointwise_lsc}
so that the generalization of the solution set of the linearized auxiliary problem in  \cref{th:auxQVI} does not have 
any impact here.
Finally, we would like to point out that, by considering the continuous and
nondecreasing function $\Phi\colon \R \to [0, \infty)$ 
given by 
\[
\Phi(s) = 1 \text{ for } s < 1,\qquad \Phi(s)  = s  \text{ for } s \in [1, 2],\qquad \Phi(s)  = 2 \text{ for } s > 2,
\]
in \eqref{eq:1dQVI}
instead of a function with the properties in \eqref{eq:concaveex1d},
one can proceed along exactly the same lines as above to 
obtain an example of a problem
\eqref{eq:FP} that is covered by \cref{thm:solutions_lipschitz} 
and satisfies $\mathbb{S}(u) = [\min(u, 1),\min(u, 2)]$, $m(u) = \min(u, 1)$, and $M(u) = \min(u, 2)$ for all $u \in [0, \infty)$.
This shows that our Lipschitz stability result also covers situations
in which the minimal solution mapping is not constant zero (as in the case of \eqref{eq:concaveex1d})
and in which the estimates \eqref{eq:mLipschitzEstimate} and 
\eqref{eq:MLipschitzEstimate} are both nontrivial. 

%%%%%%%%%%%%%%%%%%%%%%%%%%%%%%%%%%%%%%%%%%%%%%

\subsection{Application in impulse control}
\label{subsec:6.2}
Next, we consider the impulse control problem \eqref{eq:ImpVIintro}. 
Suppose that $\Omega \subset \R^d$, $d \in \mathbb{N}$, is a bounded open nonempty set, 
let $H_0^1(\Omega)$ and $H^{-1}(\Omega)$ be defined as usual (see \cite{Attouch2006}),
let $\left \langle \cdot, \cdot \right \rangle$ be the dual pairing between elements of $H^{-1}(\Omega)$
and $H_0^1(\Omega)$, let $\Delta\colon H_0^1(\Omega) \to H^{-1}(\Omega)$ denote the distributional 
Laplacian, let $\kappa \geq 0$ and $c_0 \in L_+^0(\R^d)$ be arbitrary but fixed,
and let $f\colon \R \to \R$ be a nondecreasing, globally Lipschitz continuous, convex function satisfying $f(0) = 0$. 
Given a
$u \in H^{-1}_+(\Omega) := 
\{z \in H^{-1}(\Omega) \mid \left \langle z, v\right \rangle \geq 0 \text{ for all } 0 \leq v \in H_0^1(\Omega)\}$, 
we are interested in 
the nonlinear elliptic quasi-variational inequality
\begin{equation}
\label{eq:ImpulseControlQVI}
\begin{aligned}
y \in H_0^1(\Omega), \quad~ 0 \leq y \leq \Theta(y),
\quad~ \left \langle -\Delta y + f(y) - u, v - y \right \rangle \geq 0~~\forall v \in H_0^1(\Omega), v \leq \Theta(y),
\end{aligned}
\end{equation}
where $\Theta(y)$ is the function defined by
\begin{equation}
\label{eq:ThetaDef}
\Theta(y)(x) := \kappa + \essinf_{0 \leq \xi \in \R^d,~x + \xi \in \Omega} c_0(\xi) + y(x + \xi) \quad \text{ for a.a.\ } x \in \Omega. 
\end{equation}
Here, the inequality $\xi\geq 0$ has to be understood componentwise, 
the inequalities $0 \leq v$, $0 \leq y \leq \Theta(y)$, and $v \leq \Theta(y)$ are meant to hold in the a.e.-sense,
and $f$ acts by superposition. 
Note that the above problem formulation is analogous to that considered in \cite[\mbox{section VIII-2}]{Bensoussan1982} 
(with more general $\kappa$, $c_0$, and $u$
and an additional nonlinearity~$f$).
Variations of the above setting (involving, e.g., Neumann boundary conditions)
can also be found in the literature, see  \cite{Bensoussan1975,Bensoussan1975-2,Lions1986,Perthame1984,Perthame1985}, 
and may be discussed with the same techniques that we use 
in the following. Before we start with the analysis of the solution mapping of \eqref{eq:ImpulseControlQVI},
we briefly check that the formula \eqref{eq:ThetaDef} is sensible.
To this end, we prove the measurability of a partial essential infimum.
\begin{lemma}
\label{lem:partial_essinf}
		Let $(X_1, \Sigma_1)$ be a measurable space
		and let
		$(X_2, \Sigma_2, \mu_2)$ be a $\sigma$-finite measure space.
		Assume that the function
		$F \colon X_1 \times X_2 \to [-\infty, \infty]$
		is measurable w.r.t.\ the product $\sigma$-algebra
		$\Sigma_1 \otimes \Sigma_2$.
		Then the function $G \colon X_1 \to [-\infty, \infty]$ defined via
		\[
			G(x_1)
			:=
			\essinf_{x_2 \in X_2} F(x_1, x_2)
			\qquad
			\forall x_1 \in X_1
		\]
		is measurable as well.
		The same assertion holds true if
		$(X_1, \Sigma_1, \mu_1)$ and
		$(X_2, \Sigma_2, \mu_2)$ are complete, $\sigma$-finite measure spaces
		and $F$ is measurable w.r.t.\ the completion
		of the product $\sigma$-algebra
		$\Sigma_1 \otimes \Sigma_2$
		w.r.t.\ the product measure
		$\mu_1 \otimes \mu_2$.
	\end{lemma}
	\begin{proof}
		For every $c \in \R$,
		we have
 		$M_c :=
 		F^{-1}([-\infty,c))
 		% \set{(x_1, x_2) \in X_1 \times X_2 \given F(x_1,x_2) < c}
 		\in \Sigma_1 \otimes \Sigma_2$.
		Denote the indicator function of $M_c$ with $\chi_{M_c}$.
		Then the function $h_c \colon X_1 \to [0,\infty]$
		defined via
		\[
			h_c(x_1) :=
			\int_{X_2} \chi_{M_c}(x_1, x_2) \, \d\mu_2(x_2)
			=
			\mu_2(\set{x_2 \in X_2 \given (x_1, x_2) \in M_c})\qquad
			\forall x_1 \in X_1
		\]
		is measurable by
		\cite[Proposition~5.1.3]{Cohn2013} or \cite[Satz~V.1.3]{Elstrodt2011}.
		By definition, we have $G(x_1) < c$
		if and only if $h_c(x_1) > 0$.
		Hence, the set
		$G^{-1}([-\infty,c)) = h_c^{-1}( (0,\infty] )$
		belongs to $\Sigma_1$
		for all $c \in \R$
		and, thus,
		$G$ is measurable.
		In the complete case,
		we can argue similarly by using
		\cite[Exercise~5.2.6]{Cohn2013}
		or
		\cite[Satz~V.2.4]{Elstrodt2011}.
	\end{proof}

\begin{corollary}[well-definedness of $\Theta$]%
\label{lemma:ThetaSense}%
The formula \eqref{eq:ThetaDef} yields a well-defined operator
$\Theta\colon \{v \in L^2(\Omega) \mid v \geq - \kappa\} \to L^0_+(\Omega)$. 
\end{corollary}
\begin{proof}
	Let $v \in L^2(\Omega)$ with $v \ge -\kappa$ be given.
	First of all, it is clear that
	$\Theta(v)$ does not depend on the representative of $v$.
	The measurability of $\Theta(v)$
	follows from \cref{lem:partial_essinf} applied to
	$X_1 := \Omega$,
	$X_2 := \set{\xi \in \R^d \given \xi \ge 0}$
	(both equipped with the Lebesgue $\sigma$-algebra
	and the Lebesgue measure),
	and
	\[
	F(x, \xi)
		:=
		\begin{cases}
			c_0(\xi) + v(x+\xi) & \text{if } x + \xi \in \Omega \\
			+\infty & \text{if } x + \xi \not\in \Omega
		\end{cases}
		\qquad\forall x \in \Omega, \xi \ge 0.
	\]
	From $c_0 \ge 0$ and $v \ge -\kappa$, we get $\Theta(v) \ge 0$.
	Finally, $\Theta(v) < \infty$ follows since $c_0$ and $v$ are almost everywhere finite and $\Omega$ is open.
\end{proof}

We remark that the map $\Theta$ may fail to preserve Sobolev regularity even if $c_0$ is smooth, cf.\
\cite[Remark 3-2.1]{Mosco1976}.
We omit discussing additional regularity properties of the 
functions $\Theta(v)$ because they are not needed for our analysis.
Next, we collect some classical results on the solution operator of the obstacle problem, cf.\ \cite{Rodrigues1987}.

\begin{lemma}[properties of the obstacle problem]%
\label{lemma:obstacleproperties}%
Let $H_0^1(\Omega)$ be endowed 
with the partial order $v  \geq w :\iff v \geq w~\text{a.e.\ in }\Omega$ 
(analogously to $L^q(\Omega)$, $q \in \{0\} \cup [1, \infty]$)
and 
let $H^{-1}(\Omega)$ be endowed with the partial order 
$v  \geq w :\iff \left \langle v- w, z \right \rangle \geq 0$ for all $0\leq z\in H_0^1(\Omega)$. 
Denote the function that is equal to infinity a.e.\ in $\Omega$
by $\infty$ and add it to $L^0_+(\Omega)$ as the largest element.  
Define $ U := H^{-1}_+(\Omega)$, $\bar P := L^0_+(\Omega) \cup \{\infty\}$, and $P := L^0_+(\Omega)$,
and consider for $u \in U$ and $p \in \bar P$ the problem
\begin{equation}
\label{eq:ObstacleProblem}
\begin{aligned}
y \in H_0^1(\Omega), ~~y \leq p,
\quad~~ \left \langle -\Delta y + f(y) - u, v - y \right \rangle \geq 0~~\forall v \in H_0^1(\Omega), v \leq p. 
\end{aligned}
\end{equation}
Then this nonlinear variational inequality has a unique solution $y = S(p, u) \in H_0^1(\Omega)$ for all 
$p \in \bar P$ and all $u \in U$ and the following is true: 
\begin{enumerate}
\item\label{obstacle:i}
It holds $S(p, u) \in L^2_+(\Omega)$ for all $p \in \bar P$, $u \in U$. 
\item\label{obstacle:ii} 
It holds $S(p_1, u_1) \leq S(p_2, u_2)$ for all $p_1, p_2 \in \bar P$, $u_1, u_2 \in U$, $p_1 \leq p_2$, $u_1 \leq u_2$.
\item\label{obstacle:iii}  
It holds $\lambda S(p_1, u_1) + (1 - \lambda)  S(p_2, u_2) \leq S(\lambda p_1 + (1 - \lambda)p_2, \lambda u_1 + (1 - \lambda)u_2)$ 
for all $p_1, p_2 \in P$, $u_1, u_2 \in  U$, $\lambda \in [0, 1]$.
\item\label{obstacle:iv}  
It holds $S(p, u) \in L^\infty_+(\Omega)$ for all $p \in \bar P$, $u \in L^{q}_+(\Omega)$, $q \in [1, \infty] \cap (d/2, \infty]$.
\end{enumerate}
\end{lemma}

\begin{proof}
The unique solvability of \eqref{eq:ObstacleProblem} 
for all $p \in \bar P$ and $u \in U$ follows from \cite[Theorem 4-3.1]{Rodrigues1987}.
To prove \ref{obstacle:ii}, 
let $p_1, p_2 \in \bar P$, $u_1, u_2 \in U$ with $p_1 \leq p_2$, $u_1 \leq u_2$ be given
and set $y_1 := S(p_1, u_1)$ and $y_2 := S(p_2, u_2)$. 
From Stampacchia's lemma, see \cite[Theorem 5.8.2]{Attouch2006}, 
and $p_1 \leq p_2$, it follows that $H_0^1(\Omega) \ni y_1 - \max(0, y_1 - y_2) \leq y_1 \leq p_1$
and $H_0^1(\Omega) \ni y_2 + \max(0, y_1 - y_2) = \max(y_1, y_2) \leq p_2$ holds.
This allows us to use \eqref{eq:ObstacleProblem},
the monotonicity of $f$, 
and  the inequality $u_1\leq u_2$ to obtain 
\begin{equation*}
\begin{aligned}
0 &\leq  
\left \langle -\Delta y_1 + f(y_1) - u_1, - \max(0, y_1 - y_2)  \right \rangle + \left \langle -\Delta y_2 + f(y_2)- u_2, \max(0, y_1 - y_2)  \right \rangle
\\
&\leq \left \langle - \Delta (y_2 - y_1), \max(0, y_1 - y_2)  \right \rangle.
\end{aligned}
\end{equation*}
Using \cite[Theorem 5.8.2]{Attouch2006} and the inequality of Poincar{\'e}-Friedrichs on the right-hand side 
of the last estimate yields $ \max(0, y_1 - y_2)  = 0$. Thus, $y_1 \leq y_2$ and \ref{obstacle:ii} is proved. 
Since we trivially have $S(0,0) = 0$, this also shows \ref{obstacle:i}. 
To prove \ref{obstacle:iii}, we can argue along the same lines:
If $p_1, p_2 \in P = L^0_+(\Omega)$, $u_1, u_2 \in  U$, and $\lambda \in [0, 1]$ are given 
and if $y_1 := S(p_1, u_1)$, $y_2 := S(p_2, u_2)$,
then it also holds $\lambda p_1 + (1 - \lambda)p_2 \in L^0_+(\Omega) \subset \bar P$
and  $\lambda u_1 + (1 - \lambda)u_2 \in U$.
We define $y_{12} := S(\lambda p_1 + (1 - \lambda)p_2, \lambda u_1 + (1 - \lambda)u_2)$
and $w := \lambda y_1 + (1 - \lambda)y_2 - y_{12}$.
To establish \ref{obstacle:iii}, we have to show that $w \le 0$.
From \cite[Theorem 5.8.2]{Attouch2006}, we obtain $\max(0,w) \in H_0^1(\Omega)$
and it is easy to check that $w$, $y_1$, and $y_2$ satisfy
$y_1 - \max(0, w) \leq p_1$,
$y_2 - \max(0, w) \leq p_2$, and
$y_{12} + \max(0, w) \leq \lambda p_1 + (1 - \lambda)p_2$. 
This allows us to use these functions as test functions in the 
variational inequalities satisfied by $y_1$, $y_2$, and $y_{12}$, respectively.
By multiplying the resulting inequalities for $y_1$ and $y_2$ by $\lambda$ and $(1-\lambda)$, respectively,
we arrive at the estimates 
\begin{equation*}
\begin{aligned}
\left \langle -\lambda \Delta y_1 + \lambda f(y_1)- \lambda u_1 ,  - \max(0, w) \right \rangle \geq 0,
\\
\left \langle - (1-\lambda) \Delta y_2 + (1 - \lambda)f(y_2)- (1-\lambda) u_2 ,  - \max(0, w)  \right \rangle \geq 0,
\\
\left \langle  -\Delta y_{12}  + f(y_{12})- \lambda u_1   - (1 - \lambda) u_2,  \max(0, w)  \right \rangle \geq 0.
\end{aligned}
\end{equation*}
Adding these inequalities and exploiting the convexity and monotonicity of $f$ yields 
\begin{equation*}
\begin{aligned}
	0 &\leq \dual{ -\Delta w + \lambda f(y_1) + (1 - \lambda)f(y_2) - f(y_{12})}{  - \max(0, w) } 
	\\
	&\leq \dual{ -\Delta w + f( \lambda y_1 +  (1 - \lambda) y_2) - f(y_{12})}{  - \max(0, w) } 
	\\
	&\leq \dual*{  -\Delta w}{  - \max(0, w) }.
\end{aligned}
\end{equation*}
Due to \cite[Theorem 5.8.2]{Attouch2006}, this implies $w \le 0$ as desired. 
It remains to prove \ref{obstacle:iv}.
To this end, suppose that $p \in \bar P$ and
$u \in L^{q}_+(\Omega)$, $q \in [1, \infty] \cap (d/2, \infty]$, are arbitrary but fixed, set $y:= S(p, u) \ge 0$, 
and define $y_k := y - \min(k, y) = y - \min(k, \max(y, -k))$ for all $k \geq 0$. Then 
it holds $H_0^1(\Omega) \ni \min(k, y) \leq p$ for all $k \geq 0$ and it follows straightforwardly from 
\eqref{eq:ObstacleProblem}, our assumptions on $f$, the inequality of Poincar{\'e}-Friedrichs,
and \cite[Theorem 5.8.2]{Attouch2006} that 
there exists a constant $c>0$ such that $y_k$
satisfies
\[
c \|y_k \|_{H^1(\Omega)}^2 \leq  
\left \langle -\Delta y_k ,  y_k  \right \rangle 
=
\left \langle -\Delta y ,  y_k  \right \rangle 
\leq 
\int_\Omega \left | u y_k \right | \mathrm{d}x\qquad \forall k \geq 0.
\]
From this estimate and a standard calculation involving the 
Sobolev embedding $H_0^1(\Omega) \hookrightarrow L^{r}(\Omega)$ with $r= 2d/(d-2)$ for $d>2$,
$1 \leq r < \infty$ for $d=2$, and $r=\infty$ for $d=1$
(which can be found, e.g., in \cite[Lemma 2.3]{ChristofMeyer2015} and \cite[Lemma~2.8]{Troianiello1987}), 
it follows that there exists a constant $C>0$ satisfying $\|y\|_{L^\infty(\Omega)} \leq C \|u\|_{L^q(\Omega)}$.
This shows \ref{obstacle:iv} and completes the proof. 
Note that we do not need any additional assumptions on $\Omega$ here for the Sobolev embedding to hold
due to the zero boundary conditions.
\end{proof}

Since $\Theta$ is not defined (and nonnegative) on all of $L^2(\Omega)$,
we have to truncate its argument
to be able to reformulate \eqref{eq:ImpulseControlQVI} as a problem of the type \eqref{eq:FP}.

\begin{lemma}[reformulation of the impulse control problem \eqref{eq:ImpulseControlQVI}]%
\label{lemma:impulserefomulation}%
Let $U$, $P$, $\bar P$, $S\colon \bar P \times U \to L^2_+(\Omega)$, and 
$\Theta\colon \{v \in L^2(\Omega) \mid v \geq - \kappa\} \to P$
be defined as in \cref{lemma:ThetaSense,lemma:obstacleproperties},
and denote with $\Phi$ the map
$\Phi \colon L^2(\Omega) \to P$, $v \mapsto \Theta(\max(-\kappa, v))$.
Suppose that a $u \in U$ is given. Then $y \in H_0^1(\Omega)$
is a solution of  \eqref{eq:ImpulseControlQVI} if and only if 
it is a solution of the equation 
$y = S(\Phi(y), u)$. 
\end{lemma}
\begin{proof}
If $y$ solves \eqref{eq:ImpulseControlQVI}, then it holds $y \geq 0$ 
and, by the definitions of $S$ and $\Phi$, we have $y = S(\Theta(y), u) = S(\Phi(y), u)$.
If, conversely, we start with a solution $y$ of $y = S(\Phi(y), u)$,
then it follows from $\Phi(y) \in \bar P$ and \cref{lemma:obstacleproperties}\ref{obstacle:i}
that $y \geq 0$ holds and, as a consequence, that $y  = S(\Phi(y), u) = S(\Theta(y), u)$. 
This shows that $y$ also solves \eqref{eq:ImpulseControlQVI} and completes the proof. 
\end{proof}

As the equation $y = S(\Phi(y), u)$ has precisely the form \eqref{eq:FP},
we may now use our abstract analysis to arrive at 
the following two main results of this subsection.

\begin{theorem}[concavity and pointwise directional differentiability on $H^{-1}_+(\Omega)$]%
\label{th:impulseH-1}%
The impulse control problem \eqref{eq:ImpulseControlQVI}
possesses a nonempty set of solutions $\mathbb{S}(u) \subset H_0^1(\Omega)$ 
for all $\smash{u \in H^{-1}_+(\Omega)}$ and this solution set 
possesses unique smallest and largest elements $m(u)$ and $M(u)$. 
The map $M\colon H^{-1}_+(\Omega) \to H_0^1(\Omega)$, $u \mapsto M(u)$,
is concave in the sense that 
$\lambda M(u_1) + (1 - \lambda)M(u_2) \leq M(\lambda u_1 + (1-\lambda)u_2)$ 
holds a.e.\ in $\Omega$ for all $u_1, u_2 \in H^{-1}_+(\Omega)$ and $\lambda \in [0, 1]$.
Further, $M$ is directionally differentiable in the sense that, for all  $u \in H^{-1}_+(\Omega)$
and $h \in H^{-1}(\Omega)$ satisfying $u + \tau_0 h \in H^{-1}_+(\Omega)$
for a $\tau_0 > 0$, there exists a unique $M'(u;h) \in L^0(\Omega, (-\infty, \infty])$
such that the difference quotients
\[
\frac{M(u + \tau h) - M(u)}{\tau},\qquad \tau \in (0, \tau_0),
\]
converge pointwise a.e.\ to $M'(u;h)$ along every sequence $\{\tau_n\} \subset (0, \tau_0)$ with $\tau_n \to0$.
\end{theorem}

\begin{proof}
Let $U$, $P$, $\bar P$, $S\colon \bar P \times U \to L^2_+(\Omega)$, 
$\Theta\colon \{v \in L^2(\Omega) \mid v \geq - \kappa\} \to P$,
and $\Phi \colon L^2(\Omega) \to P$ be defined as in \cref{lemma:ThetaSense,lemma:obstacleproperties,lemma:impulserefomulation}.
Then it follows from \cref{lemma:obstacleproperties}
that $U$, $P$, $\bar P$, and $S$ satisfy all of the 
conditions in \cref{subsec:2.1} and \cref{ass:DirDiff}
(with $X =\Omega$,  $\Sigma$ as the Lebesgue $\sigma$-algebra on $\Omega$,
and $\mu$ as the Lebesgue measure).
Using \eqref{eq:ThetaDef},
it is easy to check that this is also true for $\Phi$. 
Indeed, for all $v_1, v_2\in L^2(\Omega)$, $v_1 \leq v_2$, we clearly have
$\max(-\kappa,v_1) \le \max(-\kappa,v_2)$ and
$\Phi(v_1) = \Theta(\max(-\kappa,v_1)) \le \Theta(\max(-\kappa,v_2)) = \Phi(v_2)$
by the definitions of $\Theta$ and $\Phi$.
Further, \eqref{eq:ThetaDef}
implies that $\Theta$ is pointwise a.e.\ concave.
Thus,  for all $\lambda \in [0, 1]$ and all
$v_1, v_2 \in \{v \in L^2(\Omega) \mid v \geq - \kappa\}$, 
we obtain from the definition of $\Phi$ that
$
\Phi(\lambda v_1 + (1 - \lambda) v_2) 
=
\Theta(\lambda v_1 + (1 - \lambda) v_2)
\geq 
\lambda \Theta(v_1) + (1 - \lambda)\Theta(v_2)
=
\lambda \Phi(v_1) + (1 - \lambda)\Phi(v_2).
$
\Cref{th:solvability,prop:Mconcave,th:monononeM,lemma:impulserefomulation} now yield the claims. 
\end{proof}

\begin{theorem}[Lipschitz continuity and directional differentiability on $L^\infty_+(\Omega)$]
\label{th:impulsLinfty}%
On the set $\smash{L^\infty_+(\Omega) \subset  H^{-1}_+(\Omega)}$,
the minimal and maximal solution map of the impulse control problem \eqref{eq:ImpulseControlQVI}
are well-defined as operators
$m, M\colon L^\infty_+(\Omega) \to H_0^1(\Omega) \cap L^\infty_+(\Omega)$.
Further, the following is true:
\begin{enumerate}
\item\label{th:impulsLinfty:ii} 
The functions
$m$ and $M$ are locally Lipschitz continuous on $L^\infty_\oplus(\Omega)$
in the sense that, for every $u \in L^\infty_\oplus(\Omega)$,
there exist constants $C, r> 0$
such that
\begin{equation}
\label{eq:LipschitzImpuls}
\|m(v_1) - m(v_2)\|_{L^\infty(\Omega)} + \|M(v_1) - M(v_2)\|_{L^\infty(\Omega)}\leq C \|v_1 - v_2\|_{L^\infty(\Omega)}
\end{equation}
holds for all $v_1, v_2 \in  L_+^\infty(\Omega)$ satisfying $\norm{v_i- u}_{L^\infty(\Omega)} \leq r$, $i=1,2$. 

\item\label{th:impulsLinfty:iii} 
The functions
$m$ and $M$ are weakly continuous on $L^\infty_\oplus(\Omega)$ 
in the sense that, for all $u \in L^\infty_\oplus(\Omega)$ and all $\{u_n\} \subset L_\oplus^\infty(\Omega)$
with $u_n \to u$ in $L^\infty(\Omega)$, we have $m(u_n) \weakly m(u)$ and $M(u_n) \weakly M(u)$ in $H_0^1(\Omega)$.

\item\label{th:impulsLinfty:iv} 
The function $M$ is Hadamard directionally differentiable 
on the set $L_\oplus^\infty(\Omega)$ in the sense that,
for all $u \in L_\oplus^\infty(\Omega)$ and all 
$h \in L^\infty(\Omega)$, there exists a unique $M'(u; h) \in L^\infty(\Omega)$ such that,
for all  $\{\tau_n\} \subset (0, \infty)$, $\{h_n\} \subset L^\infty(\Omega)$ satisfying
$\tau_n \to 0$, $\|h - h_n\|_{L^\infty(\Omega)} \to 0$, and $u + \tau_n h_n \in L^\infty_+(\Omega)$ for all $n$, we have 
\[
\frac{M(u + \tau_n h_n) - M(u)}{\tau_n} \to M'(u; h) \text{ in } L^q(\Omega) \text{ for all } 1 \leq q < \infty
\]
and 
\[
\frac{M(u + \tau_n h_n) - M(u)}{\tau_n} \weaklystar M'(u; h) \text{ in } L^\infty(\Omega).
\]

\item\label{th:impulsLinfty:v} 
If $\kappa > 0$ holds, then \eqref{eq:ImpulseControlQVI} is uniquely solvable for all 
$u \in L^\infty_+(\Omega)$ and it holds $m \equiv M \equiv \mathbb{S}$ on $L^\infty_+(\Omega)$. 
In this situation, the directional derivative $\mathbb{S}'(u; h) = M'(u;h)$ 
of the solution operator $\mathbb{S}$ of \eqref{eq:ImpulseControlQVI}
at a point $u \in L^\infty_{\oplus}(\Omega)$ 
in a direction $h \in L^\infty(\Omega)$ satisfying $u + \tau_0 h \in L^\infty_+(\Omega)$ for a $\tau_0 > 0$ is
uniquely characterized by the condition that it is the smallest element of the set
\[
			\left \{
				\zeta \in \smash{L^{\infty}(\Omega)}
				\left |\, 
				\begin{aligned}
					&\exists  \{\zeta_n\} \subset \smash{L^{\infty}(\Omega)}, \{\tau_n\} \subset (0, \tau_0) \colon
					\\
					 &\zeta_n \leq \zeta_{n+1} \text{ and } \tau_{n+1} \leq \tau_n \text{ for all } n, 
					 \\
					& \tau_n \to 0\text{ for } n \to \infty,\;\zeta_n  \to \zeta \text{ in } L^q(\Omega) 
					\text{ for all } q \in [1, \infty),
					\\
					 &\mathbb{S}(u) +\tau_n\zeta_n \text{ is supersol.\ of } y = S(\Phi(y), u + \tau_n h)
					\text{ for all }n,
					\\
					 &\Psi'((\mathbb{S}(u),u);(\zeta_n, h)) - \zeta_n \to 0~ \text{a.e.\ in }\Omega \text{ for }n \to \infty
				\end{aligned}
				\right.
			\right \}.
\]
Here, $S$, $\Phi$, and $\Psi$ are defined as in \cref{lemma:impulserefomulation,lem:concdirdifPsi}.
\end{enumerate}
\end{theorem}
\begin{proof}
We consider the same $P$, $\bar P$, $S$, $\Theta$, and $\Phi$
as before, but this time we restrict the function
$S$ in the second argument to the set $\smash{\tilde U := L^\infty_+(\Omega) \subset H_+^{-1}(\Omega)}$. 
For this setting, we obtain from \cref{lemma:obstacleproperties} and the same arguments 
as in the proof of \cref{th:impulseH-1} that all of the conditions in \cref{subsec:2.1} and \cref{ass:DirDiff,ass:Lipschitz} are satisfied
(with $X = Y = \Omega$,  $\Sigma = \Xi$ as the Lebesgue $\sigma$-algebra on $\Omega$,
and $\mu = \eta$ as the Lebesgue measure) and, in the case $\kappa > 0$,
that \cref{ass:AuxProb} holds as well (with $\varepsilon = \kappa$). 
By invoking \cref{prop:soluniqueness,thm:solutions_lipschitz,cor:Hadamard,lemma:obstacleproperties,lemma:impulserefomulation,th:auxQVI}
and by exploiting that
$\smash{L^{[r, \infty]}(\Omega)} = L^\infty(\Omega)$ holds for all $r \in [1,\infty]$
due to the boundedness of $\Omega$, 
it now follows immediately that $m$ and $M$ 
possess the asserted mapping properties on $\smash{L^\infty_+(\Omega)}$.  
Note that \eqref{eq:LipschitzImpuls} is a consequence of \eqref{eq:randomeq273545} here
and that the weak continuity of $m$ and $M$ as functions from $L_\oplus^\infty(\Omega)$ to $H_0^1(\Omega)$
in \ref{th:impulsLinfty:iii} is obtained from the arguments outlined in 
\cref{rem:LipschitzComments}\ref{rem:LipschitzComments:iv} and the estimate 
$\smash{\|S(p,u)\|_{H^1(\Omega)} \leq C(\Omega) \|u\|_{H^{-1}(\Omega)}}$ for all $p \in \bar P$ and $u \in U$
that follows straightforwardly from \eqref{eq:ObstacleProblem} with $v=0$.
This completes the proof.
\end{proof}

We remark that, even in the case $\kappa > 0$, 
it is not possible to invoke the uniqueness result of \cite{Laetsch1975}
in the situation of \cref{th:impulseH-1}
since $S(p,u) \in L^\infty(\Omega)$ may not hold for all $p \in L^0_+(\Omega)$ and all $u \in H^{-1}_+(\Omega)$.
Similarly, in \cref{th:impulsLinfty}\ref{th:impulsLinfty:v}, 
we cannot apply the second part of \cref{th:auxQVI} to conclude that
the directional derivatives of $\mathbb{S}$ are the 
smallest elements of the ``ordinary'' solution sets of the
QVIs $\zeta = \Psi'((\mathbb{S}(u),u);(\zeta, h))$ since the 
map $\Theta$ lacks pointwise a.e.\ continuity properties w.r.t.\ convergence in the  $L^q(\Omega)$-spaces for $1 \leq q < \infty$. 
Both of these effects make the sensitivity analysis of \eqref{eq:ImpulseControlQVI} a delicate issue. 
Before we proceed,
we would like to point out that \cref{th:impulsLinfty} generalizes the result on problem \eqref{eq:ImpulseControlQVI} 
obtained in 
\cite[section~7.1.2]{Alphonse2020-1}, where the continuity of the maps
$m$ and $M$ as  functions from $L^\infty(\Omega)$ to $L^2(\Omega)$
is proved on the set $L^\infty_\oplus(\Omega)$
in the case $d=1$ and $f = 0$. 
We obtain not only continuity but even local Lipschitz continuity on $L^\infty_\oplus(\Omega)$ for all dimensions $d \in \mathbb{N}$,
for the nonlinear QVI, and for $m$ and $M$ as functions into the space $L^\infty(\Omega)$. 
Our results further establish the Hadamard directional
differentiability of $M$ on $L^\infty_\oplus(\Omega)$ into all $L^q$-spaces 
in the sense of \eqref{eq:Lqdirdiff} and \eqref{eq:weakstardirdiff}
for all $d \in \mathbb{N}$ (without any sign conditions on $h$)
and, in the case $\kappa > 0$, uniquely characterize the directional derivatives
of the solution map $\mathbb{S}$ of \eqref{eq:ImpulseControlQVI}. 
At least to the best of the authors' knowledge, \cref{th:impulsLinfty} is the first result 
in the literature to accomplish this for the problem \eqref{eq:ImpulseControlQVI}.
We are able to prove all of these properties because we do not require any restrictive assumptions on the regularity or complete continuity
 of the map $\Phi$, cf.\ \cite[Assumption 1]{Alphonse2020-1}
and also
\cite{Alphonse2019-1,Alphonse2019-2,Alphonse2020-2,Alphonse2020-4,Wachsmuth2020}.

%%%%%%%%%%%%%%%%%%%%%%%%%%%%%%%%%%%%%%%%%%%%%%

\subsection{Parabolic QVIs}%
\label{subsec:6.3}%
As a third example, we consider a parabolic QVI
with boundary controls. 
Let  $\Omega \subset \R^d$, $d \in \mathbb{N}$,
be a bounded Lipschitz domain with boundary $\partial \Omega$, let $T > 0$
and  $\psi \in L^\infty_+(\Omega)$ be given,
and let $g\colon \R \to \R$ be a globally Lipschitz continuous, bounded, nonnegative, nondecreasing
function that 
is concave on $[0, \infty)$.
Given a
$u \in L_+^\infty((0, T) \times \partial \Omega) 
\subset L^2((0, T) \times \partial \Omega) 
\cong L^2(0, T; L^2(\partial \Omega))$, we are interested in
the problem of finding a (potentially weak) solution $y$ of the variational inequality
\begin{equation}
\label{eq:parQVI}
\begin{aligned}
&y(0) = 0,\qquad H^1(\Omega) \ni y(t) \leq \psi + \Phi(y) \text{ a.e.\ in }\Omega \text{ for a.a.\ } t \in (0, T),
\\
&\int_\Omega \partial_t y(t) (v - y(t)) + \nabla y(t) \cdot \nabla (v - y(t))\mathrm{d}x 
- \int_{\partial \Omega} u(t)(v - y(t))\mathrm{d}s \geq 0
\\
&\mspace{120mu}\forall v \in H^1(\Omega),~v \leq \psi + \Phi(y) \text{ a.e.\ in } \Omega 
\text{ for a.a.\ } t \in (0, T)
\end{aligned}
\end{equation}
with $\Phi$ defined by $\Phi(y):= w(T)$ and $w$ as the solution of the heat equation
\begin{equation}
\label{eq:PhiParPDE}
\partial_t w - \Delta w = g(y) \text{ in } (0, T) \times\Omega,
\quad 
w = 0  \text{ on } (0, T)\times \partial \Omega ,
\quad 
w(0) = 0	 \text{ in } \Omega.
\end{equation}%
Here and in what follows, the appearing Lebesgue, Sobolev, and Bochner spaces $L^q(\Omega)$,
$H^1(\Omega)$, $L^2(0, T;L^2(\partial \Omega))$, etc.\ are defined as usual, see \cite{Attouch2006,Barbu1984,Heinonen2015}, 
and endowed with the canonical pointwise a.e.\ partial orders induced by the underlying measure spaces,
$\partial_t$ denotes the time derivative in the Sobolev-Bochner sense, $\nabla$ is the weak spatial gradient, 
$\Delta$ is the distributional spatial Laplacian, and $g$ acts by superposition.
We would like to emphasize that 
\eqref{eq:parQVI} is again a model problem. Other boundary conditions,
functions $\Phi$, etc.\ can be studied with the same techniques that we use in the following. 
To transform \eqref{eq:parQVI} into a problem of the type \eqref{eq:FP},
we define $X := (0, T) \times \Omega$ (endowed with the Lebesgue measure), 
$\bar P := L^\infty_+(\Omega) \cup \{\infty\}$
(endowed with the partial order induced by the a.e.-sense on $\Omega$ and with the function
that is $\infty$ a.e.\ in $\Omega$ as the 
largest element),
$P := L^\infty_+(\Omega)$,
and $U := L_+^\infty((0, T) \times \partial \Omega) $. 
Note that these $X$, $\bar P$, $P$, and $U$ satisfy all of the 
conditions in \cref{subsec:2.1} and  \cref{ass:DirDiff,ass:Lipschitz,ass:AuxProb}
(with $Y := (0, T) \times \partial \Omega$ and $\eta$
as the completion of the product measure of the Lebesgue measure on $(0, T)$ and the
surface measure on $\partial \Omega$). 
The next two lemmas establish that the same is true for the maps $S$ and $\Phi$ associated with \eqref{eq:parQVI}
and \eqref{eq:PhiParPDE}.

\begin{lemma}[properties of the parabolic obstacle problem]%
\label{lemma:parabolicobstacleprob}%
Let $P$, $\bar P$, $U$, and $X$ be defined as above
and set 
$\smash{\tilde U := \{ v\in C^{0,1}([0, T] \times \partial \Omega) \mid v \geq 0 \text{ and }v|_{\{0\} \times \partial \Omega} = 0\} \subset U}$.
Then, for all  $p \in \bar P$ and all $u \in \tilde U$, 
there exists a unique strong solution $y = S(p, u) \in {H^1(0, T;H^1(\Omega)) \cap W^{1, \infty}(0, T;L^2(\Omega))}$ 
of the parabolic variational inequality 
\begin{equation}
\label{eq:parQVI_2}
\begin{aligned}
&y(0) = 0,\qquad H^1(\Omega) \ni y(t) \leq \psi + p \text{ a.e.\ in }\Omega \text{ for a.a.\ } t \in (0, T),
\\
& \int_\Omega \partial_t y(t) (v - y(t)) + \nabla y(t) \cdot \nabla (v - y(t))\mathrm{d}x 
- \int_{\partial \Omega} u(t)(v - y(t))\mathrm{d}s \geq 0
\\
&\mspace{120mu}\forall v \in H^1(\Omega),~v \leq \psi + p \text{ a.e.\ in } \Omega 
\text{ for a.a.\ } t \in (0, T).
\end{aligned}
\end{equation}
Further, there exists a constant $C > 0$ satisfying
\begin{equation}
\label{eq:ParLipschitz}
\begin{aligned}
&\left \| S(p, u_1) - S(p, u_2)\right \|_{C([0, T];L^2(\Omega))}
+
\left \| S(p, u_1) - S(p, u_2)\right \|_{L^2(0, T;H^1(\Omega))}
\\
&\qquad\qquad\qquad\qquad \leq C \|u_1 - u_2\|_{L^2( (0, T) \times \partial \Omega)}
\quad \forall u_1, u_2 \in \tilde U~\forall p \in \bar P. 
\end{aligned}
\end{equation} 
Due to \eqref{eq:ParLipschitz}, the solution map $S$ of \eqref{eq:parQVI_2}
can be extended uniquely by continuity to a function 
 $S\colon \bar P \times U \to C([0, T];L^2(\Omega)) \cap L^2(0, T;H^1(\Omega)) \subset {L^2(0, T;L^2(\Omega)) \cong L^2(X)}$.
 This weak solution operator $S$ of \eqref{eq:parQVI_2} possesses the 
 following properties:
\begin{enumerate}
\item\label{parobstacle:i}
It holds $S(p, u) \in L^2_+(X)$ for all $p \in \bar P$, $u \in U$.
\item\label{parobstacle:ii} 
It holds $S(p_1, u_1) \leq S(p_2, u_2)$ for all $p_1, p_2 \in \bar P$, $u_1, u_2 \in U$, $p_1 \leq p_2$, $u_1 \leq u_2$.
\item\label{parobstacle:iii} 
It holds $S(p, u) \in L^\infty_+(X)$ for all $p \in P$, $u \in U$.
\item\label{parobstacle:iv}  
It holds $\lambda S(p_1, u_1) + (1 - \lambda)  S(p_2, u_2) \leq S(\lambda p_1 + (1 - \lambda)p_2, \lambda u_1 + (1 - \lambda)u_2)$
for all $p_1, p_2 \in P$, $u_1, u_2 \in  U$, and $\lambda \in [0, 1]$.
\item\label{parobstacle:v}  
It holds $\| S(p_1, u) -  S(p_2, u) \|_{L^\infty(X)} \leq \| p_1 - p_2 \|_{L^\infty(\Omega)}$
for all $p_1, p_2 \in P$, $u \in  U$.
\end{enumerate}
\end{lemma}

\begin{proof}
The existence of a unique strong solution $\smash{S(p, u) \in H^1(0, T;H^1(\Omega))} \cap \smash{W^{1, \infty}(0, T;L^2(\Omega))}$ 
of \eqref{eq:parQVI_2} for all $p \in \bar P$ and all $\smash{u \in \tilde U}$
and the estimate \eqref{eq:ParLipschitz} follow
from \cite[Theorem 4.1]{Barbu1984}
and the embeddings 
$C^{0,1}([0, T] \times \partial \Omega) \hookrightarrow H^1(0, T; L^2(\partial \Omega)) \hookrightarrow H^1(0, T;H^1(\Omega)^*)$
and $L^2((0, T) \times \partial \Omega) \cong L^2(0, T;L^2(\partial \Omega)) \hookrightarrow L^2(0, T;H^1(\Omega)^*)$. To 
see that $S$ admits a unique extension $\smash{S\colon \bar P \times U \to C([0, T];L^2(\Omega)) \cap L^2(0, T;H^1(\Omega))}$,
it suffices to note that $\smash{\tilde U}$ is dense in $U$ w.r.t.\ the norm $\smash{\|\cdot \|_{L^2((0, T) \times \partial \Omega)}}$
and to use \eqref{eq:ParLipschitz}, see \cite[Corollary 4.1]{Barbu1984}.
Suppose now that $\smash{u_1, u_2 \in \tilde U}$ and $\smash{p_1, p_2 \in \bar P}$ 
satisfying $u_1 \leq u_2$ and $p_1 \leq p_2$ are given
and define $y_1 := S(p_1, u_1)$ and $y_2 := S(p_2, u_2)$.
Then it follows from 
\cite[Lemma A.1]{ChristofVexler2020} and \cite{DWachmuth2015} 
that  $y_1 - \max(0, y_1 - y_2), y_2 + \max(0, y_1 - y_2) \in L^2(0, T;H^1(\Omega)) \cap H^1(0, T;L^2(\Omega))$ holds and 
we may use \eqref{eq:parQVI_2} and the formulas in \cite[Lemma A.1]{ChristofVexler2020} to obtain
(analogously to the elliptic case in \cref{lemma:obstacleproperties}\ref{obstacle:ii})
\begin{equation}
\label{eq:randomeq1782636}
\begin{aligned}
0 &\leq 
\int_0^\tau 
\int_\Omega \partial_t (y_2 - y_1)\max(0, y_1 - y_2)
+
\nabla (y_2 - y_1) \cdot \nabla \max(0, y_1 - y_2)
\mathrm{d}x
\\
&\qquad\qquad\qquad\qquad\qquad\qquad\qquad\quad
- \int_{\partial \Omega} 
(u_2 - u_1) \max(0, y_1 - y_2) \mathrm{d}s\,\mathrm{d}t
\\
&\leq - \frac{1}{2} \| \max(0, y_1(\tau) - y_2(\tau))\|_{L^2(\Omega)}^2 \qquad\qquad \forall \tau \in [0, T]. 
\end{aligned}
\end{equation}
The above shows that $S(p_1, u_1) \leq S(p_2, u_2)$ holds for all $u_1, u_2 \in \tilde U$ and $p_1, p_2 \in \bar P$ 
satisfying $u_1 \leq u_2$ and $p_1 \leq p_2$. 
By approximation in $L^2( (0, T) \times \partial \Omega)$,
this result readily carries over to all $u_1, u_2 \in U$ and $p_1, p_2 \in \bar P$ with 
$u_1 \leq u_2$ and $p_1 \leq p_2$, cf.\ \eqref{eq:ParLipschitz}.
(Note that the set $\tilde U$ is stable w.r.t.\ pointwise truncation and that, as a consequence, 
given $u_1, u_2 \in U$ with $u_1 \leq u_2$, it is easy to construct approximating sequences in $\tilde U$
whose elements satisfy the same inequality.)
This proves \ref{parobstacle:ii}. Since we again have $S(0, 0) = 0$ and since 
$S(p, u)(t) \leq \psi + p$ holds a.e.\ in $\Omega$ for a.a.\ $t \in (0, T)$, 
the assertions in \ref{parobstacle:i} and \ref{parobstacle:iii} follow immediately from \ref{parobstacle:ii}. 
To establish \ref{parobstacle:iv}, we can proceed completely analogously 
to the elliptic case in \cref{lemma:obstacleproperties}\ref{obstacle:iii}   
using the results in \cite[Lemma A.1]{ChristofVexler2020}, 
a calculation as in \eqref{eq:randomeq1782636}, and approximation. 
It remains to establish \ref{parobstacle:v}. To this end, 
let $u \in \tilde U$ and $p_1, p_2 \in L_+^\infty(\Omega)$ 
with associated $y_1 := S(p_1, u)$, $y_2 := S(p_2, u)$ be given
and define $z := \max(y_2 - y_1 - \|p_1 - p_2\|_{L^\infty(\Omega)}, 0)$. 
By choosing the test functions $y_1 + z$ 
and $y_2 - z$
in the variational inequalities for $y_1$ and $y_2$ and by adding, integrating, and again exploiting 
\cite[Lemma A.1]{ChristofVexler2020}, we obtain 
\[
0 \leq \int_0^\tau 
\int_\Omega \partial_t (y_1- y_2) z
+
\nabla (y_1 - y_2) \cdot \nabla z\,
\mathrm{d}x
\mathrm{d}t
=
- \int_0^\tau 
\int_\Omega (\partial_t z) z + \nabla z \cdot \nabla z\,
\mathrm{d}x
\mathrm{d}t
\]
for all $\tau \in [0, T]$. This implies $z = 0$ and,
after plugging in the definition of $z$ and switching the roles 
of $y_1$ and $y_2$,  $|y_1 - y_2| \leq  \|p_1 - p_2\|_{L^\infty(\Omega)}$ a.e.\ 
in $X$ as claimed in \ref{parobstacle:v}. 
To finally establish \ref{parobstacle:v}  for all $u \in U$, we can again use 
approximation. Note that the Lipschitz continuity estimate \eqref{eq:ParLipschitz} is sufficient for this purpose 
as the set $\{ v \in L^2(X) \mid 0 \leq v \leq \|p_1 - p_2\|_{L^\infty(\Omega)} \}$ 
is closed in $ L^2(X)$.
\end{proof}

\begin{lemma}[properties of $\Phi$]%
\label{lemma:parabolicPhiProbs}%
Let $X$ and $P$ be as before and set $\Phi(v) := w(T)$ with $w$ as the solution of the heat equation with right-hand side $g(v)$ in \eqref{eq:PhiParPDE}.
Then $\Phi$ is well-defined as a function  $\Phi\colon L^2(X) \to P$ and 
the following is true:
\begin{enumerate}
\item\label{parPhi:i}
For all $v_1, v_2 \in L^2(X)$, $v_1\leq v_2$, it holds $\Phi(v_1) \leq \Phi(v_2)$. 
\item\label{parPhi:ii}
If $g$ is concave on $[-\varepsilon, \infty)$, $\varepsilon \geq 0$,
then, for all $\lambda \in [0, 1]$ and $v_1, v_2 \in L^2(X)$ with $v_1  \geq - \varepsilon$, $v_2  \geq - \varepsilon$,
we have $\lambda \Phi(v_1) + (1 - \lambda)\Phi(v_2) \leq  \Phi(\lambda v_1 + (1 - \lambda) v_2)$. 
\item\label{parPhi:iii}
There exist an exponent $q \in [2, \infty)$ and a constant $C > 0$ such that, for all $v_1, v_2 \in L^q(X)$,
we have 
$\left \| \Phi(v_1) - \Phi(v_2)\right \|_{L^\infty(\Omega)} \leq C \|v_1 - v_2\|_{L^q(X)}$.
\end{enumerate}
\end{lemma}
\begin{proof}
Due to our assumptions on the function $g$, 
it holds $g(v) \in L_+^\infty((0, T) \times \Omega)$
for all \mbox{$v \in L^2(X)$} and, since $\Omega$
is a bounded Lipschitz domain, we obtain from \cite[Theorem~3.1]{Rehberg2015}
that there exist a $q \in [2, \infty)$ and a constant $C > 0$ such that 
the solution $w$ of \eqref{eq:PhiParPDE} with right-hand side $g(v)$ is in $C([0, T] \times \overline\Omega)$ and satisfies
\[
\|\Phi(v)\|_{L^\infty(\Omega)} 
= \|w(T)\|_{C(\overline\Omega)} \leq \|w\|_{C([0, T] \times \overline\Omega)} \leq C \|g(v)\|_{L^q(X)}
\quad \forall v \in L^2(X). 
\]
Using standard results (or again \cite[Lemma A.1]{ChristofVexler2020}), it is further 
easy to check that the solution of a heat equation with homogeneous Dirichlet boundary conditions,
the initial condition zero, and a nonnegative right-hand side is nonnegative in $[0, T] \times \overline\Omega$. 
This proves that 
$\Phi$ is indeed well-defined as a map from $L^2(X)$ into $P = L^\infty_+(\Omega)$. 
From the monotonicity and linearity of the solution operator of the heat equation
and our assumptions on $g$, one also easily obtains \ref{parPhi:i} and \ref{parPhi:ii}. 
To finally establish \ref{parPhi:iii}, we note that the global Lipschitz continuity of $g$ and again \cite[Theorem 3.1]{Rehberg2015}
 yield
\[
\|\Phi(v_1) - \Phi(v_2)\|_{L^\infty(\Omega)} \leq C \|g(v_1) - g(v_2)\|_{L^q(X)}
\leq C \|g\|_{C^{0,1}(\R)} \|v_1 - v_2\|_{L^q(X)}
\]
for all $v_1, v_2 \in L^q(X)$ with the same $C$ and $q$ as before. 
This completes the proof. 
\end{proof}

Using the results of \cref{sec:2,sec:3,sec:4,sec:5}, we now get the following for the QVI \eqref{eq:parQVI}.

\begin{theorem}[directional differentiability and Lipschitz stability for \eqref{eq:parQVI}]%
\label{the:parmain}%
The parabolic QVI \eqref{eq:parQVI} possesses 
a nonempty set of solutions
$\mathbb{S}(u) \subset C([0, T];L^2(\Omega)) \cap L^2(0, T;H^1(\Omega)) \cap L_+^\infty((0, T)\times \Omega)$
for all $u \in L_+^\infty((0, T) \times \partial \Omega) $.
This solution set 
possesses unique smallest and largest elements $m(u)$ and $M(u)$.
Further, the following is true for the maps $m,M\colon L_+^\infty((0, T) \times \partial \Omega) \to L^\infty((0, T)\times \Omega)$:
\begin{enumerate}
\item The function $M$ is concave, i.e., 
for all $u_1, u_2 \in L_+^\infty((0, T) \times \partial \Omega)$, $\lambda \in [0, 1]$, we have
$\lambda M(u_1) + (1 - \lambda)M(u_2) \leq M(\lambda u_1 + (1-\lambda)u_2)$. 

\item\label{par:item:ii} 
The functions $m$ and $M$ are locally Lipschitz continuous 
in the sense that, for all $u \in  L_\oplus^\infty((0, T) \times \partial \Omega)$,
there exist constants $C, r> 0$ with
\begin{equation}
\label{eq:parobloclip}
\begin{aligned}
\|m(v_1) - m(v_2)\|_{L^\infty((0, T)\times \Omega)} 
&+
\|M(v_1) - M(v_2)\|_{L^\infty((0, T)\times \Omega)} 
\\
&\leq C \|v_1 - v_2\|_{L^\infty((0, T)\times \partial \Omega)}
\end{aligned}
\end{equation}
for all $v_1, v_2 \in  L_+^\infty((0, T)\times \partial \Omega)$ satisfying $\|u - v_i\|_{L^\infty((0, T)\times \partial \Omega)} \leq r$,
$i=1,2$.

\item\label{par:item:iii} 
The functions $m$ and $M$ are weakly continuous 
in the sense that, for all $u, u_n \in  L_\oplus^\infty((0, T) \times \partial \Omega)$
satisfying $u_n \to u$ in $L^\infty((0, T) \times \partial \Omega)$, 
we have $m(u_n) \weakly m(u)$ and $M(u_n) \weakly M(u)$ in $L^2(0, T; H^1(\Omega))$.

\item\label{par:item:iv} 
The function $M$ is Hadamard 
directionally differentiable on $L_\oplus^\infty((0, T) \times \partial \Omega)$ in the sense that,
for all $u \in L_\oplus^\infty((0, T) \times \partial \Omega)$ and $h \in L^\infty((0, T) \times \partial \Omega) $, 
there exists a unique $M'(u; h) \in L^\infty((0, T)\times \Omega)$ such that,
for all $\{\tau_n\} \subset (0, \infty)$ and $\{h_n\} \subset L^\infty((0, T) \times \partial \Omega) $ 
satisfying $\tau_n \to 0$, $\|h - h_n\|_{L^\infty((0, T) \times \partial \Omega)} \to 0$, 
and $u + \tau_n h_n \in L_+^\infty((0, T) \times \partial \Omega)$ for all $n$, it holds 
\[
\frac{M(u + \tau_n h_n) - M(u)}{\tau_n} \to M'(u; h) \text{ in } L^q((0, T) \times \Omega) \text{ for all } 1 \leq q < \infty
\]
and 
\[
\frac{M(u + \tau_n h_n) - M(u)}{\tau_n} \weaklystar M'(u; h) \text{ in } L^\infty((0, T) \times \Omega ).
\]

\item\label{par:item:v} 
If the function $g$ is not only concave on $[0, \infty)$ but even on an interval of the form 
$[-\varepsilon, \infty)$, $\varepsilon > 0$, then the QVI \eqref{eq:parQVI} possesses 
a unique solution $\mathbb{S}(u) = m(u) = M(u)$ for all $u \in L_+^\infty((0, T) \times \partial \Omega)$. 
In this case, the derivatives $\mathbb{S}'(u; h) = M'(u; h)$ in \ref{par:item:iv} are uniquely
characterized by the condition that they are the smallest elements in $L^\infty((0, T)\times \Omega)$ of the sets 
\[
\left \{\zeta \in L^\infty((0, T)\times \Omega) \mid \zeta 
= \Psi'((\mathbb{S}(u),u);(\zeta, h))~\text{a.e.\ in }(0, T) \times \Omega\right \}.
\]
Here, $\Psi$ 
denotes the composition $\Psi(v, u) := S(\Phi(v), u)$ of the functions $S$ and $\Phi$
in \cref{lemma:parabolicobstacleprob,lemma:parabolicPhiProbs}, cf.\ \cref{lem:concdirdifPsi}.
\end{enumerate}
\end{theorem}
\begin{proof}
\Cref{lemma:parabolicobstacleprob,lemma:parabolicPhiProbs} yield that \eqref{eq:parQVI}
(or, more precisely, its  reformulation $y = S(\Phi(y), u)$)
satisfies all of the conditions in \cref{subsec:2.1} and 
\cref{ass:Lipschitz,ass:DirDiff}
(with 
$X$, 
$\bar P$, $P$, 
$U$,
$Y$, $S$, and $\Phi$ as before). 
They further show that there exists a $q \in [2, \infty)$ such that
the function ${L^q(X) \ni v \mapsto \Psi(v, u) := S(\Phi(v), u) \in  L^\infty_+(X)}$
is continuous for every arbitrary but fixed $u \in U$
and that \cref{ass:AuxProb} holds
when $g$ is concave on $[-\varepsilon, \infty)$ for some $\varepsilon > 0$.
By combining all of this with 
\cref{th:solvability,prop:Mconcave,thm:solutions_lipschitz,cor:Hadamard,prop:soluniqueness,th:auxQVI},
the assertions of the theorem follow immediately. 
Note that, to obtain \eqref{eq:parobloclip} and the weak continuity in \ref{par:item:iii}, one can use \eqref{eq:ParLipschitz} 
and the arguments outlined in points \ref{rem:LipschitzComments:ii} and 
\ref{rem:LipschitzComments:iv} of  \cref{rem:LipschitzComments}.
\end{proof}

Coupled parabolic systems of the type \eqref{eq:parQVI}
arise, for instance, in the context of thermoforming, 
see \cite{Alphonse2019-1,Alphonse2019-2,Aubin1979,Prigozhin1996-2,Prigozhin1996-1}.
Note that \cref{the:parmain} again does not require any assumptions 
on the sign of the directions $h$ or
on the operator norms of $\Phi$ and its derivatives,
cf.\ \cite{Alphonse2020-2}. At least to the authors' best knowledge,
this is the first differentiability result for parabolic QVIs in such a general setting.
The same seems to be the case for the characterization of the derivatives $\mathbb{S}'(u; h)$ in 
\cref{the:parmain}\ref{par:item:v}.

%%%%%%%%%%%%%%%%%%%%%%%%%%%%%%%%%%%%%%%%%%%%%%

\bibliographystyle{siamplain}
\bibliography{references}

\begin{thebibliography}{10}

\bibitem{Adly2010}
{\sc S.~Adly, M.~Bergounioux, and M.~Ait~Mansour}, {\em Optimal control of a
  quasi-variational obstacle problem}, J.~Global Optim., 47 (2010),
  pp.~421--435, \href{https://doi.org/10.1007/s10898-008-9366-y}{doi:
  \nolinkurl{10.1007/s10898-008-9366-y}}.

\bibitem{Alphonse2019-1}
{\sc A.~Alphonse, M.~Hinterm{\"u}ller, and C.~N. Rautenberg}, {\em Directional
  differentiability for elliptic quasi-variational inequalities of obstacle
  type}, Calc.~Var.~PDE, 58 (2019),
  \href{https://doi.org/10.1007/s00526-018-1473-0}{doi:
  \nolinkurl{10.1007/s00526-018-1473-0}}.
\newblock Art.~39.

\bibitem{Alphonse2019-2}
{\sc A.~Alphonse, M.~Hinterm{\"u}ller, and C.~N. Rautenberg}, {\em Recent
  trends and views on elliptic quasi-variational inequalities}, in Topics in
  Applied Analysis and Optimisation, M.~Hinterm{\"u}ller and J.~F. Rodrigues,
  eds., Cham, 2019, Springer, pp.~1--31,
  \href{https://doi.org/10.1007/978-3-030-33116-0_1}{doi:
  \nolinkurl{10.1007/978-3-030-33116-0_1}}.

\bibitem{Alphonse2020-2}
{\sc A.~Alphonse, M.~Hinterm{\"u}ller, and C.~N. Rautenberg}, {\em Existence,
  iteration procedures and directional differentiability for parabolic {QVIs}},
  Calc.~Var.~PDE, 59 (2020),
  \href{https://doi.org/10.1007/s00526-020-01732-6}{doi:
  \nolinkurl{10.1007/s00526-020-01732-6}}.
\newblock Art.~95.

\bibitem{Alphonse2020-4}
{\sc A.~Alphonse, M.~Hinterm{\"u}ller, and C.~N. Rautenberg}, {\em On the
  differentiability of the minimal and maximal solution maps of elliptic
  quasi-variational inequalities}, 2020,
  \href{https://arxiv.org/abs/2009.01626}{arXiv: 2009.01626}.

\bibitem{Alphonse2020-1}
{\sc A.~Alphonse, M.~Hinterm{\"u}ller, and C.~N. Rautenberg}, {\em Stability of
  the solution set of quasi-variational inequalities and optimal control}, SIAM
  J.~Control Optim., 58 (2020), pp.~3508--3532,
  \href{https://doi.org/10.1137/19m1250327}{doi:
  \nolinkurl{10.1137/19m1250327}}.

\bibitem{Attouch2006}
{\sc H.~Attouch, G.~Buttazzo, and G.~Michaille}, {\em Variational Analysis in
  Sobolev and BV Spaces}, MPS/SIAM Series on Optimization, SIAM, Philadelphia,
  2006, \href{https://doi.org/10.1137/1.9781611973488}{doi:
  \nolinkurl{10.1137/1.9781611973488}}.

\bibitem{Aubin1979}
{\sc J.-P. Aubin}, {\em Mathematical Methods of Game and Economic Theory},
  North-Holland, 1979.

\bibitem{Barbu1984}
{\sc V.~Barbu}, {\em Optimal Control of Variational Inequalities}, Pitman,
  1984.

\bibitem{Bensoussan1982}
{\sc A.~Bensoussan}, {\em Stochastic Control by Functional Analysis Methods},
  North-Holland, 1982.

\bibitem{Bensoussan1975}
{\sc A.~Bensoussan and J.~L. Lions}, {\em Nouvelles methodes en contr\^{o}le
  impulsionnel}, Appl.~Math.~Optim., 1 (1975), pp.~289--312,
  \href{https://doi.org/10.1007/bf01447955}{doi:
  \nolinkurl{10.1007/bf01447955}}.

\bibitem{Bensoussan1975-2}
{\sc A.~Bensoussan and J.~L. Lions}, {\em Optimal impulse and continuous
  control: method of nonlinear quasi-variational inequalities}, Trudy
  Mat.~Inst.~Steklov., 134 (1975), pp.~5--22.

\bibitem{Brokate2015}
{\sc M.~Brokate and P.~Krejči}, {\em Weak differentiability of scalar
  hysteresis operators}, Discrete Contin.~Dyn.~Syst., 35 (2015),
  pp.~2405--2421.

\bibitem{Christof2019parob}
{\sc C.~Christof}, {\em Sensitivity analysis and optimal control of
  obstacle-type evolution variational inequalities}, SIAM J.~Control Optim., 57
  (2019), pp.~192--218, \href{https://doi.org/10.1137/18m1183662}{doi:
  \nolinkurl{10.1137/18m1183662}}.

\bibitem{ChristofMeyer2015}
{\sc C.~Christof and C.~Meyer}, {\em Differentiability properties of the
  solution operator to an elliptic variational inequality of the second kind},
  {Ergebnisberichte angewandte Mathematik, TU Dortmund, Nr.\ 527}, 2015.

\bibitem{ChristofVexler2020}
{\sc C.~Christof and B.~Vexler}, {\em New regularity results and finite element
  error estimates for a class of parabolic optimal control problems with
  pointwise state constraints}, ESAIM Control Optim.~Calc.~Var., 27 (2021),
  \href{https://doi.org/10.1051/cocv/2020059}{doi:
  \nolinkurl{10.1051/cocv/2020059}}.
\newblock Art.~4.

\bibitem{ChristofWachsmuth2020}
{\sc C.~Christof and G.~Wachsmuth}, {\em Differential sensitivity analysis of
  variational inequalities with locally {L}ipschitz continuous solution
  operators}, Appl.~Math.~Optim., 81 (2020), pp.~23--62,
  \href{https://doi.org/10.1007/s00245-018-09553-y}{doi:
  \nolinkurl{10.1007/s00245-018-09553-y}}.

\bibitem{Cohn2013}
{\sc D.~L. Cohn}, {\em Measure Theory}, Springer, 2013,
  \href{https://doi.org/10.1007/978-1-4614-6956-8}{doi:
  \nolinkurl{10.1007/978-1-4614-6956-8}}.

\bibitem{Dietrich2001}
{\sc H.~Dietrich}, {\em Optimal control problems for certain quasivariational
  inequalities}, Optimization, 49 (2001), pp.~67--93,
  \href{https://doi.org/10.1080/02331930108844521}{doi:
  \nolinkurl{10.1080/02331930108844521}}.

\bibitem{Rehberg2015}
{\sc K.~Disser, A.~F.~M. ter Elst, and J.~Rehberg}, {\em {H\"older} estimates
  for parabolic operators on domains with rough boundary},
  Ann.~Sc.~Norm.~Super.~Pisa Cl.~Sci., XVII (2017), pp.~65--79,
  \href{https://doi.org/10.2422/2036-2145.201503_013}{doi:
  \nolinkurl{10.2422/2036-2145.201503_013}}.

\bibitem{Elstrodt2011}
{\sc J.~Elstrodt}, {\em {Ma\ss- und Integrationstheorie}}, Springer, 7th
  revised and updated~ed., 2011,
  \href{https://doi.org/10.1007/978-3-642-17905-1}{doi:
  \nolinkurl{10.1007/978-3-642-17905-1}}.

\bibitem{Fremlin2010}
{\sc D.~H. Fremlin}, {\em Measure Theory.~{V}ol.~2}, Torres Fremlin,
  Colchester, 2nd~ed., 2010.

\bibitem{Haraux1977}
{\sc A.~Haraux}, {\em How to differentiate the projection on a convex set in
  {H}ilbert space. {S}ome applications to variational inequalities},
  J.~Math.~Soc.~Japan, 29 (1977), pp.~615--631,
  \href{https://doi.org/10.2969/jmsj/02940615}{doi:
  \nolinkurl{10.2969/jmsj/02940615}}.

\bibitem{Heinonen2015}
{\sc J.~Heinonen, P.~Koselka, N.~Shanmugalingam, and J.~T. Tyson}, {\em Sobolev
  Spaces on Metric Measure Spaces}, Cambridge University Press, 2015,
  \href{https://doi.org/10.1017/cbo9781316135914}{doi:
  \nolinkurl{10.1017/cbo9781316135914}}.

\bibitem{Laetsch1975}
{\sc T.~Laetsch}, {\em A uniqueness theorem for elliptic quasi-variational
  inequalities}, J.~Funct.~Anal., 18 (1975), pp.~286--287,
  \href{https://doi.org/10.1016/0022-1236(75)90017-8}{doi:
  \nolinkurl{10.1016/0022-1236(75)90017-8}}.

\bibitem{Lions1986}
{\sc P.~L. Lions and B.~Perthame}, {\em Quasi-variational inequalities and
  ergodic impulse control}, SIAM J.~Control Optim., 24 (1986), pp.~604--615,
  \href{https://doi.org/10.1137/0324036}{doi: \nolinkurl{10.1137/0324036}}.

\bibitem{Mignot1976}
{\sc F.~Mignot}, {\em Contrôle dans les inéquations variationelles
  elliptiques}, J.~Funct.~Anal., 22 (1976), pp.~130--185,
  \href{https://doi.org/10.1016/0022-1236(76)90017-3}{doi:
  \nolinkurl{10.1016/0022-1236(76)90017-3}}.

\bibitem{Mosco1976}
{\sc U.~Mosco}, {\em Implicit variational problems and quasi variational
  inequalities}, in Nonlinear Operators and the Calculus of Variations, J.~P.
  Gossez, E.~J. Lami~Dozo, J.~Mawhin, and L.~Waelbroeck, eds., Berlin,
  Heidelberg, 1976, Springer, pp.~83--156,
  \href{https://doi.org/10.1007/BFb0079943}{doi:
  \nolinkurl{10.1007/BFb0079943}}.

\bibitem{Perthame1984}
{\sc B.~Perthame}, {\em Quasi-variational inequalities and
  {H}amilton-{J}acobi-{B}ellman equations in a bounded region}, Comm.~PDE, 9
  (1984), pp.~561--595, \href{https://doi.org/10.1080/03605308408820342}{doi:
  \nolinkurl{10.1080/03605308408820342}}.

\bibitem{Perthame1985}
{\sc B.~Perthame}, {\em Some remarks on quasi-variational inequalities and the
  associated impulsive control problem}, Annales de l'I.~H.~P., Section~C, 2
  (1985), pp.~237--260,
  \href{https://doi.org/10.1016/s0294-1449(16)30404-8}{doi:
  \nolinkurl{10.1016/s0294-1449(16)30404-8}}.

\bibitem{Prigozhin1996-2}
{\sc L.~Prigozhin}, {\em On the {B}ean critical-state model in
  superconductivity}, Eur.~J.~Appl.~Math., 7 (1996), pp.~237--247,
  \href{https://doi.org/10.1017/s0956792500002333}{doi:
  \nolinkurl{10.1017/s0956792500002333}}.

\bibitem{Prigozhin1996-1}
{\sc L.~Prigozhin}, {\em Variational model of sandpile growth},
  Eur.~J.~Appl.~Math., 7 (1996), pp.~225--235,
  \href{https://doi.org/10.1017/s0956792500002321}{doi:
  \nolinkurl{10.1017/s0956792500002321}}.

\bibitem{Rodrigues1987}
{\sc J.~Rodrigues}, {\em Obstacle Problems in Mathematical Physics},
  North-Holland, 1987.

\bibitem{Troianiello1987}
{\sc G.~M. Troianiello}, {\em Elliptic Differential Equations and Obstacle
  Problems}, The University Series in Mathematics, Plenum Press, New York,
  1987, \href{https://doi.org/10.1007/978-1-4899-3614-1}{doi:
  \nolinkurl{10.1007/978-1-4899-3614-1}}.

\bibitem{DWachmuth2015}
{\sc D.~Wachsmuth}, {\em The regularity of the positive part of functions in
  {$L^2(I; H^1(\Omega)) \cap H^1(I; H^1(\Omega)^*)$} with applications to
  parabolic equations}, Comment.~Math.~Univ.~Carolin., 57 (2016), pp.~327--332.

\bibitem{Wachsmuth2019}
{\sc G.~Wachsmuth}, {\em A guided tour of polyhedric sets: basic properties,
  new results on intersections and applications}, J.~Convex Anal., 26 (2019),
  pp.~153--188.

\bibitem{Wachsmuth2020}
{\sc G.~Wachsmuth}, {\em Elliptic quasi-variational inequalities under a
  smallness assumption: uniqueness, differential stability and optimal
  control}, Calc.~Var.~PDE, 59 (2020),
  \href{https://doi.org/10.1007/s00526-020-01743-3}{doi:
  \nolinkurl{10.1007/s00526-020-01743-3}}.
\newblock Art.~82.

\end{thebibliography}

\end{document}